\documentclass[reqno, 12pt]{amsart}

\usepackage{enumerate,amsmath,amssymb, fullpage}
\usepackage{hyperref}

\theoremstyle{plain}

\makeatletter
\newtheorem*{rep@theorem}{\rep@title}
\newcommand{\newreptheorem}[2]{%
\newenvironment{rep#1}[1]{%
 \def\rep@title{#2 \ref{##1}}%
 \begin{rep@theorem}}%
 {\end{rep@theorem}}}
 \makeatother

\newtheorem{thm}{Theorem}[section] 
\newreptheorem{theorem}{Theorem}
\newtheorem{theorem}[thm]{Theorem} 

\newtheorem{corollary}[thm]{Corollary} 
\newtheorem{lem}[thm]{Lemma} 
\newtheorem{lemma}[thm]{Lemma}
 
\newtheorem{proposition}[thm]{Proposition} 

\newtheorem{conjecture}[thm]{Conjecture}

\theoremstyle{definition}

\newtheorem{definition}[thm]{Definition} 
\newtheorem{example}[thm]{Example}
\newtheorem{remark}[thm]{Remark}

\renewcommand{\epsilon}{\varepsilon}

\let\theta\vartheta
\let\phi\varphi

\let\<\langle
\let\>\rangle

\DeclareMathAlphabet{\doba}{U}{msb}{m}{n}
\gdef\mC{\doba{C}}

\gdef\mN{\doba{N}}

\gdef\R{\doba{R}}

\gdef\SS{\mathbb{S}}

\renewcommand{\i}{{\rm i}}
\def\ran{{\mathop{\rm ran}}}
\def\dom{{\mathop{\rm dom}}}

\def\Spin{{\mathop{\rm Spin}}}
\def\spin{{\mathop{\rm Spin}}}

\def\supp{{\mathop{\rm supp}}}

\def\Id{\operatorname{Id}}

\newcommand{\s}{{\rm scal}}
\newcommand{\bq}{\begin{equation}}
\newcommand{\eq}{\end{equation}}
\newcommand{\Spinc}{\mathrm{Spin^c}}

\newcommand{\definedas}{\mathrel{\raise.095ex\hbox{\rm :}\mkern-5.2mu=}}

\begin{document}

\title 
[]
{Boundary value problems for noncompact boundaries of Spin$^c$ manifolds and spectral estimates}

\author{Nadine Grosse} 
\address{Nadine Grosse \\ 
Mathematisches Institut \\ 
Universit\"at Leipzig \\
04009 Leipzig \\
Germany}
\email{grosse@math.uni-leipzig.de}

\author{Roger Nakad} 
\address{Roger Nakad \\ 
Notre Dame University-Louaiz\'e \\ 
Faculty of Natural and Applied Sciences\\
Department of Mathematics and Statistics\\
P.O. Box 72, Zouk Mikael\\
Lebanon}
\email{rnakad@ndu.edu.lb}


\date{\today}


\thanks{The work was initiated during the  Junior Trimester 
Program in Differential Geometry at the Hausdorff Research Institute for Mathematics (HIM) in Bonn where the first author participated. The authors gratefully acknowledge the support of the HIM. The second author thanks the Institute of Mathematics of the University of Leipzig 
for its support and hospitality. The first author is very much indebted to Bernd Ammann, Ulrich Bunke and Cornelia Schneider for many enlightening  discussions. We also want to thank the referee for very helpful comments.}

\begin{abstract} 
We study boundary value problems for the  Dirac operator on Riemannian $\Spinc$ manifolds of bounded geometry and with 
noncompact boundary. This generalizes a part of the theory of boundary value problems by Ch. 
B\"ar and W. Ballmann for complete manifolds with closed boundary. As an application, we derive the lower bound of Hijazi-Montiel-Zhang, 
involving the mean curvature of the boundary, 
for the spectrum of the Dirac operator 
on the noncompact  boundary of a $\Spinc$ manifold, and the limiting case is studied. 
\end{abstract}
\maketitle
\section{Introduction}
In the last years, the spectrum of the Dirac operator on hypersurfaces of $\Spin$ 
manifolds has been intensively studied.  Indeed, many extrinsic upper bounds have been obtained  
(see \cite{Am1, Am2, AF, An, Ba, Bm} and references therein) and more recently 
in \cite{HMZ1, HMZ2, HMZ02, HZ1, HZ2, Z}, extrinsic lower bounds for the hypersurface Dirac operator are established. 
From these spectral estimates and their limiting cases, many topological and 
geometric informations on the hypersurface are derived.\\\\
In \cite{HMZ1}, O. Hijazi, S. Montiel and X. Zhang investigated the spectral
 properties of the Dirac operator on a compact manifold with  boundary for the Atiyah-Patodi-Singer type boundary
 condition (or shortly APS-boundary condition) corresponding to the spectral resolution of the classical Dirac operator 
of the boundary hypersurface. They proved that, on  the compact boundary  $\Sigma = \partial M$ of a compact 
Riemannian $\spin$ manifold $(M^{n+1}, g)$ of nonnegative
scalar curvature $\s^M$, the first nonnegative eigenvalue 
of the Dirac operator on the boundary satisfies
\begin{eqnarray}\label{hmz}
 \lambda_1 \geq \frac n2 \inf_\Sigma H,
\end{eqnarray}
where the mean curvature of the boundary $H$ is calculated with respect to the
inner normal and assumed to be nonnegative. Equality holds in \eqref{hmz} if and only if $H$ is
constant and every eigenspinor associated with the eigenvalue
 $\lambda_1$ is the restriction to $\Sigma$ of a parallel spinor field on $M$
(and hence $M$ is Ricci-flat). As application of the limiting case, they gave an elementary $\Spin$ proof of
the famous Alexandrov theorem: {\it The only closed embedded hypersurface in $\R^{n+1}$ of constant mean
curvature is the sphere of dimension $n$.}\\\\
Furthermore, Inequality \eqref{hmz} does not only give an extrinsic lower bound on the
first nonnegative eigenvalue but can
 also be seen as an obstruction to positive scalar curvature of the interior
given only in terms of a neighbourhood of the boundary. More precisely, let a
neighbourhood of the boundary $\Sigma$ be equipped with a metric of
nonnegative scalar curvature and such that the boundary has nonnegative mean curvature. If the lowest positive 
eigenvalue of the Dirac operator on the boundary is smaller than
$\frac{n}{2}\inf_{\Sigma} H$, then the metric cannot be extended to all of
$M$ such that the scalar curvature remains nonnegative.\\

In this paper, we extend the lower bound \eqref{hmz} to noncompact boundaries of Riemannian
$\Spinc$ manifolds under suitable geometric assumptions, see Theorem \ref{main}. When shifting from the compact case to the noncompact case, many obstacles
occur. Moreover, when shifting from the classical $\Spin$ geometry to $\Spinc$ geometry, the situation is more general 
since the spectrum of the Dirac operator will not only depend on the geometry of the manifold but also on the connection of
 the auxiliary line bundle associated with the fixed $\Spinc$ structure.\\

When we consider a Riemannian $\Spin$ or $\Spinc$ manifold with noncompact boundary, the main technical difference to the compact case 
is that we cannot
restrict all our computations to smooth spinors. For compact manifolds, this is
possible by using the spectral decomposition of $L^2$ by an eigenbasis. For
complete manifolds, eigenspinors do not have to exist or even if they do, in
general they do not form an orthonormal basis of $L^2$ since continuous
spectrum can occur. Additionally, the proof of Inequality \eqref{hmz} in the closed case uses 
the existence of a solution of a boundary value problem defined under the $APS$-boundary condition. While for noncompact 
boundaries 
 the idea 
of $APS$-boundary conditions can be carried over to noncompact boundaries
 by using the spectral theorem, it is not clear to us whether they actually define an actual boundary condition, 
see Example \ref{ex_bd_cond}.\\

In order to circumvent all these problems,  a large part of the paper is devoted to give a generalization of the theory 
of boundary value problems for noncompact boundaries, see Section \ref{boundary_values}. We stick
to the part of the theory that gives existence of solutions of such boundary value problem, cf. Remark \ref{comparebb}.
For complete manifolds with closed boundary, the theory of boundary value problems is given in \cite{baer_ballmann_11} by Ch.~B\"ar and W.~Ballmann. They did not only restrict to the classical Dirac operator but they generalized the traditional 
theory of elliptic boundary value problems to Dirac type operators. Additionally, they  proved a decomposition 
theorem for the essential spectrum, a general version of Gromov and Lawson's relative index 
theorem and a generalization of the cobordism theorem.\\\\
In Section \ref{boundary_values}, we will classify boundary conditions for a Riemannian $\Spinc$ manifold $(M^{n+1}, g)$ with 
noncompact boundary  $\Sigma := \partial M$ and of bounded geometry, see Definition \ref{bdd_geo}. Indeed, we prove in Section \ref{boundary_values} that  
the trace map or the restriction map $R: \phi\mapsto \phi|_\Sigma$ where $\phi$ is a compactly supported smooth spinor on 
$M$ can be extended to a bounded operator 
 $$R: \dom\, D_{\mathrm{max}} \to H_{-\frac{1}{2}} (\Sigma, \SS_M|_\Sigma).$$ Here $\dom\, D_\text{max}$ is the maximal
 domain of the Dirac operator on $M$, $\SS_M|_\Sigma$ is the restriction of the $\Spinc$ bundle $\SS_M$ to $\Sigma$ 
and for $H_{-\frac{1}{2}} (\Sigma, \SS_M|_\Sigma)$ see Definition \ref{-12-Sobolev}. The map $R$ is not surjective. But in Theorem \ref{workaround_image}, we show that 
there is an extension map $\tilde{\mathcal{E}}$ -- a right inverse to the restriction map $R: \Gamma_c^\infty(M, \SS_M) \to \Gamma_c^\infty(M, \SS_M)$ -- such that $\tilde{\mathcal{E}}R$ is a bounded linear operator 
from $\dom\, D_\text{max}$ to itself. The definition of $\tilde{\mathcal{E}}$ uses the extension map for closed boundaries introduced by B\"ar and Ballmann in \cite{baer_ballmann_11} as local building blocks. This will allow  to 
equip $R(\dom\ D_\mathrm{max})$ with a norm $\Vert.\Vert_{\check{R}}$ that turns it into a Hilbert space. With these ingredients, we can then classify the closed extensions of the Dirac operator $D_{cc}$ acting on smooth compactly 
supported spinors on $M$: For every closed extension of the Dirac operator acting on smooth compactly supported spinors on $M$ the set $B:=R(\dom\, D)\subset H_{-\frac{1}{2}} (\Sigma, \SS_M|_\Sigma)$ is closed  in $( R(\dom\, D_{\mathrm{max}}), \Vert.\Vert_{\check{R}})$. Conversely, every closed linear subset  $B\subset (R(\dom\, D_{\mathrm{max}}),\Vert.\Vert_{\check{R}})$ gives the domain  $\dom\, D_B$ of a closed extension.
Such subsets $B$ are called a boundary conditions.\\\\
Then,  we generalize the existence result for boundary value problems to our noncompact setting. For this, we need the
 notion of {\it $B$-coercivity at infinity}, see Definition \ref{coer}. This notion generalizes the
 notion of {\it coercivity at infinity} for closed boundaries as used in \cite{baer_ballmann_11}, where this assumption is also needed when characterizing the Fredholmness of the Dirac operator.  The {\it 
$B$-coercivity at infinity} condition will in 
general depend on the boundary condition $B$ and under some additional assumptions, it coincides with the 
{\it coercivity at infinity} condition used in \cite{baer_ballmann_11}. 

\begin{theorem}\label{intro-bvp} Let $M$ be a Riemannian $\Spinc$ manifold with boundary $N$. Let $(M,N)$ and the auxiliary line bundle $L$ over $M$ be of bounded geometry, cp. Definitions~\ref{bdd_geo} and~\ref{bdd_geo2}. Let $B\subset R(\dom\, D_{\rm max})$ be a boundary condition, and let the Dirac operator 
$$D_B\colon \mathrm{dom} D_B\subset L^2(M, \SS_M) \to L^2(M, \SS_M)$$
be $B$-coercive at infinity. Let $P_B$ be a projection from $R(\dom\, D_\mathrm{max})$ to $B$. 
Then, for all $\psi\in L^2(M, \SS_M)$ and $\tilde{\rho}\in \dom\, D_\mathrm{max}$ 
 where $\psi-D\tilde{\rho}\in (\ker (D_{B})^*)^\perp$  the boundary value problem
$$\left\{
\begin{array}{rlr}
 D\phi &=\psi & \text{on}\ M,\\
(\Id - P_{B})R\phi&= (\Id - P_{B})R\tilde{\rho}&\text{on}\ \Sigma,
\end{array}
\right.
$$
has a unique solution $\phi\in \dom\, D_\mathrm{max}$, up to elements of the kernel $\ker D_B$.
\end{theorem}

Note that projection just means a linear operator that restricted to $B$ acts as identity operator. 

Theorem \ref{intro-bvp}  will be one of the main ingredients to generalize Inequality \eqref{hmz} to our noncompact setting. As boundary condition $B$ we will not take the APS-boundary condition as in the closed case but another one: $B_\pm$, cf. Section \ref{boundcond_B+-}. For closed boundaries, the $B_\pm$ boundary condition
 was introduced in \cite{HMZ02} to prove a conformal version of \eqref{hmz}. 
Using Theorem~\ref{intro-bvp} for the boundary condition $B_\pm$ and the  $\Spinc$ Reilly 
inequality on possibly open boundary domains,  we obtain
\begin{theorem}\label{main}
Let $(M^{n+1},g)$ be a complete Riemannian $\Spinc$ manifold  with boundary $\Sigma$ and  $L$ be the auxiliary line bundle associated to the $\Spinc$-structure. Assume that $(M,\Sigma)$ and $L$ are of
bounded geometry. Moreover, we assume that
  $\Sigma$ has nonnegative
 mean curvature $H$ with respect to its inner unit normal field of $\Sigma$, 
the Dirac operator $D$ is $(B_+)$- or $(B_-)$-coercive at infinity and that $ \s^M
+2\i\Omega\cdot$ is a nonnegative operator where $\i\Omega$ denotes the curvature $2$-form of $L$.  Then, the 
infimum $\lambda_1$ of the nonnegative part of the spectrum of the Dirac operator on $\Sigma$ satisfies
\begin{eqnarray*}
\lambda_1\geq \frac{n}{2} \inf_{\Sigma} H.
\end{eqnarray*}
If $\lambda_1 \geq 0$ is an eigenvalue, equality holds if and only if $H$ is constant and any eigenspinor corresponding
 to $\lambda_1$ is the restriction of a parallel $\Spinc$ spinor  $\phi$ on $M$.
\end{theorem}

The paper is structured as follows: In Section \ref{prelim},  we give all the preliminaries as e.g. 
the $\Spinc$ Dirac operator and the assumption on the bounded geometry. In Section \ref{sec_trace} we review the trace and extension theorem for Sobolev spaces on manifolds of bounded geometry and appropriate noncompact boundary, the spectral decomposition of the Dirac operator on the boundary and analyze an extension map for the maximal domain of the Dirac operator. 
The theory of boundary values will be generalized to
 our noncompact setting in Section \ref{boundary_values}. The special boundary condition $B_\pm$ needed to proof
 the desired inequality is examined in Section \ref{boundcond_B+-}. In Section \ref{section_coer}, we study the coercivity 
condition for the Dirac operator. Then, we review the spinorial Reilly inequality in order to ready to proof 
the inequality in Section \ref{proofmain}.

\section{Notations and preliminaries}\label{prelim}

In this section, we briefly review 
some basic facts about $\Spinc$ geometry. Then, we give the necessary
preliminaries on Sobolev spaces on manifolds with boundary, the Trace
Theorem and its implications, some basics of spectral theory, and we recall the closed
range theorem.\\

\paragraph{\bf The $\Spinc$ Dirac operator.}  Let $(M^{n+1}, g)$ be an $(n+1)$-dimensional
Riemannian $\Spinc$ manifold with boundary. On such a manifold we have a
Hermitian complex
vector bundle $\SS_M$ endowed with a natural scalar product $\langle., .\rangle$ and with  a connection 
$\nabla$ which parallelizes the metric. 
Moreover, the bundle $\SS_M$, called the $\Spinc$ bundle, is endowed with  a Clifford
multiplication denoted by ``$\cdot$'', 
$\cdot : TM \longrightarrow  \mathrm{End}_\mC (\SS_M)$, such that at every
point $x\in M$, ``$\cdot$''defines 
an irreducible
representation of the corresponding Clifford algebra. Hence, the complex rank of
$\SS_M$ is $2^{[\frac{n+1}{2}]}$. Given a $\Spinc$ structure on $(M^{n+1}, g)$, one can prove that the
determinant line bundle $\mathrm{det}\ \SS_M$ has a root of index $2^{[\frac
{n+1}{2}]-1}$, see \cite[Section 2.5]{6}. We denote 
by $L$ this root
line bundle over $M$ and call it the auxiliary line bundle associated with the
$\Spinc$ structure.

Locally, a $\Spin$ structure always exists. We denote by $\SS_M^{'}$ the (possibly globally non-existent)
 spinor bundle. Moreover, the square root of the
 auxiliary line bundle $L$ always exists locally. But, $\SS_M = \SS_M^{'} \otimes L^{\frac 12}$, see 
\cite[Appendix D]{6} and
 \cite{nakadthese}.  This essentially means  that, while the spinor bundle and
 $L^{\frac 12}$  may not exist globally, their tensor product (the $\Spinc$ bundle) is defined globally. Thus, 
the connection $\nabla$ on $\SS_M$ is the twisted connection of the one on the spinor bundle (coming 
from the Levi-Civita connection) and a fixed connection on $L$.\\

We denote by $\Gamma_c^\infty(M, \SS_M)$ the set of all compactly supported  smooth spinors on $M$. This allows boundary values if
$\partial M \neq \emptyset$.  
The set of smooth spinors that are compactly supported in the interior of $M$ is
denoted by $\Gamma_{cc}^\infty(M, \SS_M)$. For abbreviation, we set $L^2=L^2(M)=L^2(M, \SS_M)$ and $L^2(\Sigma)=L^2(\Sigma, \SS_M|_\Sigma)$ and 
analogously for other function spaces. Moreover, $(.,.)$ shall always denote the $L^2$-scalar product on $M$ 
and $(.,.)_\Sigma$ the one on $\Sigma$.

With these ingredients, we may define the Dirac operator $D$ acting on the space of smooth sections of $\SS_M$ -- denoted by $\Gamma^\infty(M, \SS_M)$ -- by the composition of the metric connection and the Clifford
multiplication. In local coordinates this reads as
$$D =\sum_{j=1}^{n+1} e_j \cdot \nabla_{e_j}$$
where $\{e_j\}_{j=1,\cdots, n+1}$ is an orthonormal basis of $TM$. It is a first-order elliptic
operator satisfying for all smooth spinors $\varphi,\psi$ on $M$ at least one of them being compactly supported \
\begin{align}\label{L2-structure_mod_boundary}
(D\psi, \varphi)-(\psi, D\varphi)=-\int_{{\partial M}} \langle \nu\cdot\psi|_{\partial M},
\varphi|_{\partial M}\rangle ds,
\end{align}
where $(., .)$ is the $L^2$-scalar product given by $(\phi, \psi)=\int_M \langle
\phi, \psi\rangle dv$, $\partial M$ is the boundary of $M$, $|_{ \partial M}$ denotes the restriction to the boundary, $\nu$ the inner unit
normal 
vector of the embedding $\partial M \hookrightarrow M$,  and $dv$ (resp. $ds$) is
the Riemannian volume form of $M$ 
(resp. of $\partial M$). Hence, if $\partial M = \emptyset$, the Dirac operator
is formally self-adjoint with respect to the $L^2$-scalar product.\\

An important tool when examining the Dirac operator on $\Spinc$ manifolds is the
Schr\"{o}dinger-Lichnerowicz formula:
\begin{eqnarray}
D^2 = \nabla^*\nabla + \frac 14 \s^M\; \mathrm{Id}_{\Gamma (\SS_M)}+
\frac{\i}{2}\Omega\cdot,
\label{sl}
\end{eqnarray}
where $\nabla^*$ is the adjoint of $\nabla$ with respect to the $L^2$-scalar
product, 
$\i\Omega$ is the curvature  of the auxiliary line bundle $L$ associated with a
fixed connection ($\Omega$ is a real $2$-form 
on $M$) and $\Omega\cdot$ is the extension of the Clifford multiplication to
differential forms. 
\begin{example}
(i) A $\Spin$ structure can be 
seen as a $\Spinc$ structure with trivial auxiliary
line bundle $L$ and trivial connection (and so $\i\Omega =0$).\\
(ii) Every almost complex manifold $(M^{2m = n+1}, g, J)$ 
of complex dimension $m$ has a canonical $\Spinc$ structure. In fact, the complexified cotangent bundle 
$T^*M\otimes \mathbb{C} = \Lambda^{1,0} M \oplus \Lambda^{0,1}M$ 
decomposes into the $\pm \i$-eigenbundles of the complex linear extension of the complex structure $J$. 
Thus, the spinor bundle of the canonical $\Spinc$ structure is given by $$\SS_M = \Lambda^{0,*} M =\oplus_{r=0}^m \Lambda^{0,r}M,$$
where $\Lambda^{0,r}M = \Lambda^r(\Lambda^{0,1}M)$ is the bundle of $r$-forms of type $(0, 1)$. 
The auxiliary line bundle of this canonical $\Spinc$ structure is given by  
$L = (K_M)^{-1}= \Lambda^m (\Lambda^{0,1}M)$, where $K_M$ is the canonical bundle of $M$ \cite{6, 19, HMU1, nakadthese}. 
Let $\alpha$  be the K\"{a}hler form defined by the complex structure $J$, i.e. $\alpha (X, Y)= g(X, JY)$ for 
all vector fields $X,Y\in \Gamma(TM).$ The auxiliary line bundle $L= (K_M)^{-1}$ has a canonical holomorphic
 connection induced from the Levi-Civita connection whose curvature form is given by $\i\Omega= \i\rho$, where $\rho$
 is the Ricci $2$-form given by $\rho(X, Y) = \mathrm{Ric} (X, JY)$. Here $\mathrm{Ric}$ denotes the Ricci tensor of $M$. For any other $\Spinc$ structure on $M^{2m}$, the spinorial bundle  can be written as \cite{6, HMU1}:
$$\SS_M  = \Lambda^{0,*}M\otimes\mathcal L,$$
where $\mathcal L^2 = K_M\otimes L$ and $L$  is the auxiliary bundle associated with this $\Spinc$
structure. In this case, the $2$-form $\alpha$ can be considered as an endomorphism of $\SS_M$ via
 Clifford multiplication and we have the well-known orthogonal splitting $\SS_M = \oplus_{r=0}^{m}\SS_M^r,$
where $\SS_M^r$ denotes the eigensubbundle corresponding 
to the eigenvalue $\i (m-2r)$ of $\alpha$, with complex rank $\binom{m}{k}$. The bundle $\SS_M^r$ corresponds
to $\Lambda^{0, r}M\otimes\mathcal L$. For the canonical $\Spinc$ structure, the subbundle $\SS_M^0$ is trivial. 
Hence and when $M$ is a  K\"{a}hler manifold, this $\Spinc$ structure admits parallel spinors (constant functions) 
lying in $\SS_M^0$ \cite{19}. Of course, we can define another $\Spinc$ structure for which the spinor bundle is 
given by 
$\Lambda^{*, 0} M =\oplus_{r=0}^m \Lambda^r (T_{1, 0}^* M)$ and the auxiliary line bundle by $K_M$. 
This $\Spinc$ structure is called the anti-canonical $\Spinc$ structure.\\
\end{example}
Any $\Spinc$ structure on $(M^{n+1}, g)$ induces a $\Spinc$ structure on its
boundary $\Sigma = \partial M$ and we  have
$$\left\{
\begin{array}{l}
{\SS_M}{|_\Sigma} \ \ \ \ \simeq \SS_{\Sigma}\ \text{\ \ \ if\ $n$ is even,} \\
{\SS^+_M}{|_\Sigma}\ \ \ \ \simeq\SS_{\Sigma} \ \text{\ \ \ if\ $n$ is odd.}
\end{array}
\right.
$$
We recall that if $n$ is odd, the spinor bundle $\SS_M$ splits into  
 $$\SS_M = {\SS^+_M} \oplus {\SS^-_M},$$
by the action of the complex volume element. Moreover, Clifford multiplication
with a vector field $X$ tangent to $\Sigma$ is given by 
$$X\bullet\phi = (X\cdot \nu\cdot \psi){|_\Sigma},$$
where $\psi \in \Gamma^\infty(M, \SS_M)$ (or $\psi \in \Gamma^\infty(\SS^+_M)$ if $n$ is odd),
$\phi$ is the restriction of $\psi$
 to $\Sigma$, '$\bullet$' is the Clifford multiplication on $M$. 
When $n$ is odd we also get ${\SS^-_M}  \simeq
\SS_{\Sigma}$. In this case, the Clifford multiplication by a 
vector field $X$ tangent to $\Sigma$ 
is given by $X\bullet\phi = - (X\cdot\nu\cdot\psi){|_\Sigma}$ and hence we
have 
${\SS_M}{|_\Sigma} \simeq \SS_{\Sigma} \oplus\SS_{\Sigma}$. 
Moreover, the corresponding auxiliary line bundle $L^\Sigma$ on $\Sigma$ is the restriction to $\Sigma$ of
the auxiliary line bundle $L$ 
and $\i\Omega^\Sigma = {\i\Omega}{|_\Sigma}$. 
We denote by $\nabla^\Sigma$ the spinorial Levi-Civita connection on
$\SS_{\Sigma}$.
 For all smooth vector fields $X\in \Gamma^\infty(T\Sigma)$ and for every smooth spinor field $\psi \in
\Gamma^\infty(M, \SS_M)$, we consider $\phi= \psi{|_\Sigma}$
 and  we have the following $\Spinc$ Gauss formula \cite{HMU1, nakadthese,
2ana}:
\begin{equation*}
(\nabla_X\psi){|_\Sigma} =  \nabla^{\Sigma}_X \phi + \frac 12 II(X)\bullet\phi,
\end{equation*}
where $II$ denotes the Weingarten map with respect to $\nu$. Moreover, let $D$
and $D^\Sigma$ be the Dirac operators 
on $M$ and $\Sigma$. After denoting any smooth spinor and its
restriction to $\Sigma$  by the same symbol, we have on $\Sigma$ (see \cite{HMU1, 2ana, nakadthese}) that 
\begin{eqnarray}
\widetilde{D}^{\Sigma} \phi = \frac{n}{2}H\phi -\nu\cdot D\phi-\nabla_{\nu}\phi ,
\label{diracgauss}
\end{eqnarray}
\begin{eqnarray}
\widetilde{D}^{\Sigma}(\nu\cdot\varphi) = -\nu\cdot\widetilde{D}^{\Sigma} \varphi,
\label{d1}
\end{eqnarray}
where $H = \frac 1n \mathrm{tr}(II)$ denotes the mean curvature and 
$\widetilde{D}^{\Sigma} = D^\Sigma$ if $n$ is even and $\widetilde{D}^{\Sigma}=D^\Sigma
\oplus(-D^\Sigma)$ if $n$ is odd. Note that $\sigma(\widetilde{D}^{\Sigma})=\{\pm\lambda\ |\ \lambda\in\sigma(D^\Sigma)\}$ 
where $\sigma(A)$ denotes the spectrum of an operator $A$.\\

\paragraph{\bf Bounded geometry.} In this paragraph, we recall the definition of  manifolds of bounded geometry.   
\begin{definition}\label{bdd_geo}\cite[Definition 2.2]{Schick01}
Let $(M^{n+1},g)$ be a complete Riemannian manifold with boundary $\Sigma$. We say
that $(M,\Sigma)$ is of bounded geometry if the following is fulfilled
\begin{itemize}
\item[(i)] The curvature tensor of $M$ and all its covariant derivatives are bounded.
\item[(ii)] The injectivity radius of $\Sigma$ is positive.
 \item[(iii)] There is a collar around $\Sigma$, i.e: There is $r_\partial>0$ such that the geodesic collar 
\[F: U_\Sigma=[0,r_\partial)\times \Sigma \to M,\ (t,x)\mapsto \exp_x(t\nu)\]
is a diffeomorphism onto its image where $\nu$ is the inner unit normal field on $\Sigma$. We equip $U_\Sigma$ with the induced metric and will identify $U_\Sigma$ with its image.
 \item[(iv)] There exists $\epsilon>0$ such that the injectivity radius of each point $x\in M\setminus U_\Sigma$ is greater or equal than $\epsilon$.
 \item[(v)] The mean curvature of $\Sigma$ and all its covariant derivatives are bounded.
\end{itemize}
\end{definition}
\begin{definition}\label{bdd_geo2}(cp. \cite[A.1.1]{Shubin} together with \cite[Theorem B]{Eich})
 Let $E$ be a hermitian vector bundle over $M$ where $(M,\Sigma)$ is of bounded geometry. Then $E$ is said to be of bounded geometry 
if its curvature and all its covariant derivatives are bounded.
\end{definition}
\begin{remark}
\begin{enumerate}
\item Note that the above definition contains the usual definition of manifold of bounded
geometry without boundary. Moreover, if $(M,g)$ is of bounded geometry,
 then $(\Sigma, g|_\Sigma)$ is also of bounded geometry \cite[Corollary 2.24]{Schick01}.
\item For the spinor bundle $\SS_M^{'}$ associated with a $\Spin$ structure, the bounded geometry follows
automatically from the bounded geometry of $M$, \cite[Section 3.1.3]{Ammha}. For a $\Spinc$ manifold the 
situation is more general since the $\Spinc$ bundle  $\SS_M$ does not only depend
 on the geometry of the underlying manifold but also on the geometry of the auxiliary line bundle $L$. But,  
$\SS_M = \SS_M^{'}\otimes L^\frac{1}{2}$, where $\SS_M^{'}$
 is the locally defined spinor bundle, $L^{\frac{1}{2}}$ is locally defined too and $\SS_M$ is globally defined. 
Thus, the assumption that $L$ is of bounded geometry assures that $\SS_M$ is also of bounded geometry.\end{enumerate}
\end{remark}
\framebox{
{\bf Assumption for the rest of the paper:}  $(M,\Sigma)$ and $L$  are of bounded geometry.}\\

\paragraph{\bf The Sobolev space $H_1$ on manifolds with boundary.} 
We define the $H_1 = H_1 (M, \SS_M)$-norm on $\Gamma_{c}^\infty(M, \SS_M)$ by 
 \[\Vert \phi\Vert_{H_1(M, \SS_M)}^2 = \Vert \phi\Vert_{L^2(M, \SS_M)}^2 + \Vert
 \nabla\phi\Vert_{L^2(M, \SS_M)}^2.\]
 Finally, we define  $H_1=H_1 (M, \SS_M)$ as the closure  of $\Gamma_c^\infty(M,
\SS_M)$ with respect to the $H_1$-norm defined above. 

Using the Lichnerowicz formula \eqref{sl},
the Gau\ss\ theorem $(\nabla^*\nabla \phi, \phi)= \Vert
\nabla\phi\Vert^2_{L^2}+\int_{\Sigma} \langle \nabla_\nu\phi,\phi\rangle ds$, \eqref{L2-structure_mod_boundary} and
\eqref{diracgauss}, we obtain another description of the $H_1$-norm: For all $\phi\in \Gamma_c^\infty(M, \SS_M)$,  we have \begin{eqnarray}\label{Lich}
\Vert \phi\Vert_{H_1}^2 = \Vert \phi\Vert_{L^2}^2 + \Vert
D\phi\Vert_{L^2}^2-\int_{M} \frac{\s^M}{4}|\phi|^2 dv - \int_{M} \frac {i}{2}
\<\Omega\cdot\phi, \phi \> dv +\int_{{\Sigma}} \langle \phi|_\Sigma, D^W (\phi|_\Sigma) \rangle ds, 
\end{eqnarray}
where $D^W = \widetilde{D}^{\Sigma} -\frac n2 H$ is the so-called Dirac-Witten operator. Note that due to the local 
expression of $D$ and the Cauchy Schwarz inequality,  we always have 
\begin{equation}\Vert D\phi\Vert_{L^2}^2 \leq \int_M \left( \sum_{i=1}^{n+1}  |\nabla_{e_i} \phi|\right)^2 dv\leq (n+1)\Vert \nabla\phi\Vert_{L^2}^2,
\label{equ_H1D_easydir}\end{equation}
for all $\phi\in H_1(M, \SS_M)$.\\

\paragraph{\bf Spectral theory.} Most of the following can be found in \cite{baer}. In this paragraph, we shortly
review the spectral theory of the Dirac  operator $D\colon H_1(N, \SS_N)\subset L^2(N,
\SS_N) \to L^2(N, \SS_N)$ on a complete Riemannian $\Spinc$ manifold $N$ without boundary. 
Note that we assume that $N$ is of bounded geometry, and hence the graph norm of $D$, $\Vert.\Vert_D$, and the $H_1$-norm are equivalent. Then $D$ is self-adjoint and the spectrum is real. A real number $\lambda$ is an eigenvalue of $D$ if there exists a nonzero
spinor $\varphi \in H_1$ with $D\varphi =  \lambda \varphi$. Then $\varphi$ is
called an eigenspinor to the eigenvalue $\lambda$. Standard local elliptic regularity
theory gives that an eigenspinor is always smooth. The set of all eigenvalues is
denoted by $\sigma_p(D^\Sigma)$ -- the point spectrum. If $N$ is closed, the Dirac operator has a pure point spectrum.
But on open manifolds, the spectrum might have a continuous part. In general,
the spectrum -- denoted by $\sigma(D)$ --
is composed of the point, the continuous and the residual
spectrum. In case of a self-adjoint operator -- as we have --  there is no
residual spectrum.  Often another decomposition of the spectrum is used -- the one into discrete
spectrum
$\sigma_d(D)$ and essential spectrum $\sigma_{ess}(D)$. 
A real number $\lambda$ lies in the essential spectrum of $D$ if there
exists a sequence
of smooth compactly supported spinors $\varphi_i$ which $\Vert \varphi_i\Vert_{L^2}=1$, $\varphi_i$ converge weakly to zero and 
$$
\Vert (D - \lambda )\varphi_i \Vert_{ L^2}  \longrightarrow 0.$$
The essential spectrum contains amongst other elements all eigenvalues of infinite multiplicity. In
contrast, the
discrete spectrum $\sigma_{d}(D) := \sigma_{p}(D) \smallsetminus
\sigma_{ess}(D)$ consists of all eigenvalues of finite multiplicity.\\

\paragraph{\bf Closed Range Theorem.} Next, we want to recall briefly (a part of) the Closed Range Theorem for later use.
\begin{thm}\label{closed_range_theorem}\cite[p.205]{Yo} Let $T: X\to Y$ be a
closed linear operator between Banach spaces $X,Y$. Then the range $\ran(T)$ of
$T$ is closed in $Y$ if and only if $\ran (T)=\ker (T^*)^\perp$ where $T^*$ is
the adjoint operator of $T$ and $\ker (T^*)$ is the kernel of $T^*$.
\end{thm}
A linear operator $T: X\to Y$ between Banach spaces is called Fredholm  if its
kernel is finite dimensional and its image has
 finite codimension.

\section{Trace theorems and extensions}\label{sec_trace}

We consider the restriction operator
\begin{eqnarray*}
R: \Gamma_c^\infty(M, \SS_{M}) &\to&
\Gamma_c^\infty(\Sigma, \SS_{M}|_\Sigma)\\
\phi&\mapsto& \phi|_{\Sigma}.
\end{eqnarray*}

 If it is clear from the
context that $R\phi$ is considered instead of $\phi$, we will sometimes abbreviate
by using $\phi$ only. 
 The first part of this section will be devoted to see how the restriction operator $R$ extends to a bounded linear operator between the Sobolev spaces $H_1(M, \SS_M)$ and $H_{\frac{1}{2}} (\Sigma, \SS_M|_\Sigma)$. This Theorem is known as Trace Theorem and is a very classical result for $\R^n_+$ and compact manifolds with boundary. After reviewing the Euclidean result and basic definitions, we will shortly review how this result extends to manifolds $(M,\Sigma)$ of bounded geometry. In particular, the restriction operator will have a bounded linear right inverse -- that is called extension operator $\mathcal{E}$.

For more details on the definition of bounded geometry on manifolds with boundary see \cite{Schick01}. For the equivalence of all those different definitions of Sobolev-norms involved here and the corresponding theorems for submanifolds (not necessarily hypersurfaces) see \cite{conny}. 

For our purpose, Sobolev spaces will not be sufficient later on. The maximal domain of the Dirac operator is bigger than $H_1(\SS_M)$. The rest of this section is devoted to define an extension operator $\tilde{\mathcal{E}}$ such that $\tilde{\mathcal{E}}R: \Gamma^\infty_{c}(M, \SS_M)\to \Gamma^\infty_{c}(M, \SS_M)$ extends to a bounded operator w.r.t. the graph norm of $D$. For the definition of $\tilde{\mathcal{E}}$ we will use the special extension map introduced by B\"ar and Ballmann in \cite{baer_ballmann_11} for closed boundaries.

\subsection{Trace and Extension for Sobolev spaces}

\paragraph{\bf Trace Theorem for functions on $\mathbb{R}_{+}^{n+1}=\{ (x_0,x_1,\ldots, x_n)\in \mathbb{R}^{n+1}\ |\ x_0\geq 0\}$.} 

We identify the boundary of $\R_+^{n+1}$ with $\R^n$. First we repeat the definition of the Sobolev spaces $H_s(\R^n, \mC^r)$:

\begin{definition}\label{-12-Sobolev-euclidean}\cite[Definition 3.1]{taylor_81}
Let $s\in \mathbb{R}$. The $H_s:=H_{s}^2$-norm of a compactly supported function
$f:\mathbb{R}^n \mapsto \mathbb{C}^r$ is defined as
\[ \Vert f\Vert_{H_s(\mathbb{R}^n, \mathbb{C}^r)}^2:= \int_{\mathbb{R}^n} \left|\hat{f}(\xi)\right|^2 (1+|\xi|)^s d \xi\] where
$\hat{f}(x):=(2\pi)^{-\frac{n}{2}} \int_{\mathbb{R}^n} e^{-ix\cdot \xi} f(\xi)
d\xi$ denotes the Fourier transform of $f$. The space $H_s(\mathbb{R}^n, \mathbb{C}^r)$ is then defined as the completion of
$\Gamma_c^\infty(\mathbb{R}^n, \mathbb{C}^r)$, the space of smooth compactly supported functions on $\mathbb{R}^n$ with values in $\mathbb{C}^r$, with respect to the $H_s$-norm.
\end{definition}

The spaces $H_s(\R^{n+1}_+, \mathbb{C}^r)$ are defined analogously.

\begin{thm}\, \cite[p.138, Remark 1]{TF},\cite[Theorem I.3.4]{taylor_81}, \cite[Theorem 7.34 and 7.36]{RenRog}  Let $s>\frac{1}{2}$. The restriction map for complex valued smooth functions $R: \Gamma_c^\infty(\R^{n+1}_+)\to \Gamma_c^\infty(\R^{n})$, $f\to f|_{\R^n}$ extends to a bounded linear operator from $H_s(\R^{n+1}_+)$ to $H_{(s-\frac{1}{2})} (\R^{n})$. Moreover there is an extension operator $\mathcal{E}: H_{(s-\frac{1}{2})} (\R^{n})\to H_s(\R^{n+1}_+) $ that is a bounded linear operator and a right inverse to $R$.
\end{thm}

The generalization of this theorem to vector-valued Sobolev spaces follows immediately by the following:  Let $f=(f_1, \ldots, f_r): \mathbb{R}^n\to
\mathbb{C}^r$. Then the norms $\Vert f\Vert_{H_{s}(\mathbb{R}^n, \mathbb{C}^r)}$
and $\sum_{i=1}^r \Vert f_i\Vert_{H_{s}(\mathbb{R}^n, \mathbb{C})}$ are
equivalent.\\

\paragraph{\bf Trace Theorem on manifolds of bounded geometry.}

From now on, let $M$ be a Riemannian manifold possibly with boundary and of
bounded geometry, as in Definition \ref{bdd_geo}. Moreover, let $E$ be a
hermitian vector bundle over $M$. We assume that $E$ is also of bounded
geometry, see Definition \ref{bdd_geo2}. 
In order to obtain a trace theorem for sections in $E$ we need coordinates of the manifold that are adapted to the structure of the boundary. Those will be Fermi coordinates and there will be a adapted synchronous trivialization of $E$. This will allow that we can use the trace theorem on $\R^n$ on the individual charts to obtain the trace theorem on $(M,\Sigma)$.

In the following, we restrict to trace theorems for Sobolev spaces over $L^2$, for more general domains as Sobolev spaces over $L^p$ or Triebel-Lizorkin spaces see \cite{conny}.

Before we define Sobolev spaces for sections of $E$, we introduce Fermi coordinates adapted to the boundary and a corresponding synchronous trivialization of the vector bundle:

\begin{definition}\cite[Definition 4.3 and Lemma 4.4]{conny},\label{def_synch}\cite[Definition 2.3]{Schick01}
Let $(M^n,\Sigma)$ be of bounded geometry, see Definition \ref{bdd_geo} and the notions defined therein. 

Let $r=\min\{ \frac{1}{2} r_{\Sigma}, \frac{1}{4}r_M, \frac{1}{2}r_\partial\}$ where $r_{\Sigma}$ is the injectivity radius of ${\Sigma}$ and $r_M$ the one of $M$. Let  $p^{\Sigma}_\alpha\in {\Sigma}$ and $p_\beta\in M$ be points such that 
\begin{itemize}
 \item  the metric balls $B_{r}^{\Sigma}(p^{\Sigma}_\alpha)$ in ${\Sigma}$ (i.e. w.r.t. the metric $g|_{\Sigma}$) give a uniformly locally finite cover of ${\Sigma}$
 \item  the metric balls $B_{r}(p_\beta)$ in $M$ cover $M\setminus U_r(\Sigma)$ where $U_r(\Sigma):=F([0,r)\times \Sigma)$ and those balls are uniformly locally finite on all of $M$.
\end{itemize}
Let $(U_\gamma)_{\gamma}$ be a locally finite covering of $M$ where each $U_\gamma$ is of the form $B_{r}(p_\beta)$ or $U^{\Sigma}_{p_\alpha^{\Sigma}}=F([0,2r) \times B^{\Sigma}_{2r} (p_\alpha^{\Sigma}))$. By construction the covering $(U_\gamma)_{\gamma}$ is locally finite. Coordinates on $U_\gamma$ are chosen to be geodesic normal coordinates around $p_\beta$ in case $U_\gamma=B_{r}(p_\beta)$. Otherwise coordinates are given by Fermi coordinates \[ \kappa_\alpha: U^{\Sigma}_{p_\alpha^{\Sigma}}:=[0, 2r) \times B_{2r}(0)\subset \mathbb{R}^{n} \to U^{\Sigma}_{p_{\alpha}^{\Sigma}},\quad  (t,x)\mapsto \exp_{\exp^{\Sigma}_{p_{\alpha}^{\Sigma}}(x)} (t\nu)\] where $\nu$ is the inner normal field of $\Sigma$ and $\exp^{\Sigma}$ is the exponential map on ${\Sigma}$ w.r.t. the induced metric. 
 We call such coordinates $(U_\gamma, \kappa_\gamma)_\gamma$ Fermi coordinates for $(M,\Sigma)$. If $U_\gamma=B_r(p_\gamma)$, $E|_{U_\gamma}$  is trivialized via parallel transport along radial geodesic and we
identify $E|_{U_\gamma}$ with the trivial $\mathbb{C}^r$-bundle over $U_\gamma$.
Otherwise, $E|_{U_\gamma}$ is trivialized via parallel transport along radial geodesic of the boundary and along the normal direction. The obtained trivialization is denoted by $(\xi_\gamma)_\gamma$.
\end{definition}

In case of manifolds without boundary,  the Definition of $\xi_\gamma$ in  \ref{def_synch} is the usual definition 
of synchronous trivialization as found in 
\cite[Section 3.1.3]{Ammha}. Note that by construction the restriction of a synchronous trivialization of $E$ over a manifold $M$ to its boundary $\Sigma$ gives
a synchronous trivialization of $E|_{\Sigma}$.

\begin{lem}\label{part_un}\cite[Lemma 4.8]{conny} There is a partition of unity $h_\gamma$ subordinated to the Fermi coordinates introduced above fulfilling:  For all $k\in \mN$ there is $c_k>0$  such that for all $\gamma$ and all multi-indices $\mathfrak{a}=(\mathfrak{a}_1,\ldots, \mathfrak{a}_n)$ with $|\mathfrak{a}|\leq k$ 
\[|D^\mathfrak{a} (h_\gamma\circ \kappa_\gamma)|\leq c_k.\]
Here, $D^\mathfrak{a}= \frac{\partial^{\mathfrak{a}_1}}{(\partial x_1)^{\mathfrak{a}_1}}\cdots \frac{\partial^{\mathfrak{a}_n}}{(\partial x_n)^{\mathfrak{a}_n}}$ where $x_i$ are the coordinates.
\end{lem}

Now we have all the ingredients to define Sobolev spaces on $E$ via local pullback to vector valued functions over $\mathbb{R}^n$:

\begin{definition}\label{-12-Sobolev}\cite[Definition 5.9]{conny}
Let $s\in \mathbb{R}$. Let $(U_\gamma)_{\gamma}$ be a covering of $M$ together with a synchronous trivialization $\xi_\gamma$ of $E$ as
defined above. 
Moreover, let the covering be locally finite, and let $h_\gamma$ be a partition of
unity subordinated to $U_\gamma$ as in Lemma \ref{part_un}.
Then 
\[ \Vert \phi \Vert_{H_s(M,E)}^2:=\sum_{\alpha} \Vert \xi_{\alpha*}(h_\alpha
\phi)\Vert_{H_s(\mathbb{R}_+^n,\mathbb{C}^r)}^2.\]
\end{definition}

Note that up to equivalence the $H_s$-norm does not depend on the choices of $(U_\gamma, h_\gamma, \xi_\gamma)$, cp. \cite[Theorem 4.9, 5.11 and Lemma 5.13]{conny}.

\begin{remark}\label{rem_app} 

\begin{itemize}
\item[(i)] For $s\in \mathbb{N}$ the definition of $H_s(M,E)$ from above
is equivalent to the usual definition given by 
\[ \Vert \phi \Vert_{H_s(M,E)}:=\sum_{i=0}^s \Vert \underbrace{\nabla^E\cdots
\nabla^E}_{i\ \mathrm{times}} \phi\Vert_{L^2(M,E)},\]
cp. \cite{Schick01}, \cite[Theorem 5.7]{conny}.
\item[(ii)] For $s\leq t$ we have $\Vert \phi\Vert_{H_s(M,E)} \leq \Vert \phi\Vert_{H_t(M,E)}$. That is seen for $M=\mathbb{R}^n_+$
immediately using $(1+|\xi|)^s\leq (1+|\xi|)^t$. For general $M$,  one just lifts this result by using a partition of unity and a synchronous
trivialization.
\item[(iii)] Let $D^\Sigma: \Gamma_c^\infty (\Sigma, \SS_{\Sigma}) \to
\Gamma_c^\infty (\Sigma, \SS_{\Sigma})$ be the Dirac operator on $\Sigma$. For any
$s\in \mathbb{R}$, there is a unique closed extension of $D^\Sigma$ from $H_{s}(\Sigma,
\SS_\Sigma) \to H_{s-1}(\Sigma, \SS_\Sigma)$. 
\end{itemize}
\end{remark}

\begin{theorem}\label{trace_theorem}
Let $M^{n}$ be a Riemannian manifold with boundary $\Sigma$. Assume that $(M, \Sigma)$
is of bounded geometry and that $E$ is a hermitian vector bundle over $M$ that
is also of bounded geometry. Then, for all $s\in\mathbb{R}$ with $s> \frac{1}{2}$ the operator $R:
\Gamma_c^\infty(M, E) \to \Gamma_c^\infty(\Sigma, E|_\Sigma)$ with $\phi\mapsto \phi|_\Sigma$
extends to a bounded linear  operator from $H_{s}(M, E)$ to $H_{s-\frac{1}{2}}(\Sigma,
E|_\Sigma)$. Moreover,  there is a bounded right inverse $\mathcal{E}: H_{s-\frac{1}{2}}(\Sigma, E_\Sigma)\to H_s(M,E)$ of the trace map $R:  H_s(M,E)\to  
  H_{s-\frac{1}{2}}(\Sigma, E|_\Sigma)$. In particular, $\mathcal{E} (\Gamma_c^\infty(\Sigma, E|_\Sigma)) \subset \Gamma_c^\infty(M, E_M)$.
\end{theorem}

\begin{proof} This theorem is a special case of \cite[Theorem 5.14]{conny}. We shortly sketch the basic idea:
We choose a covering $U_\gamma$ together with a synchronous trivialization $\xi_\gamma$ of $E$ and
a subordinated partition of unity $h_\gamma$ as in Definition \ref{def_synch} and Lemma \ref{part_un}. The restrictions
$U_\gamma\cap \Sigma$ then cover $\Sigma$. 
Let $\phi\in H_s(M,E)$. Then, for all $\alpha$ we have
$\xi_{\alpha*}(h_\alpha \phi)\in H_s(\mathbb{R}_+^n, \mathbb{C}^r)$. Thus, there
exists a $C>0$ with $\Vert R(\xi_{\gamma*}(h_\gamma
\phi))\Vert_{H_{s-\frac{1}{2}}(\mathbb{R}^{n-1}, \mathbb{C}^r)}\leq C \Vert
\xi_{\gamma*}(h_\gamma \phi)\Vert_{H_{s}(\mathbb{R}_+^n, \mathbb{C}^r)}$. 

With 
 $R(\xi_{\alpha*}(h_\alpha \phi))=\xi_{\alpha*}(h_\alpha R\phi)$ 
we get  after summing up that  $\Vert R\phi\Vert_{H_{s-\frac{1}{2}}(\Sigma, E|_\Sigma)}\leq C \Vert
\phi\Vert_{H_{s}(M,E)}$ since $\xi_\alpha$ is still a synchronous trivialization
for $E|_\Sigma$.

The rest is proven analogously as the Trace Theorem using the original Euclidean version of the extension map $\mathcal{E}:  H_{s-\frac{1}{2}}(\R^{n-1})\to H_s(\R^n)$. The last inclusion follows immediately from $\mathcal{E} (\Gamma_c^\infty(\R^{n-1})) \subset \Gamma_c^\infty(\R^n)$. \end{proof}

The last theorem gives immediately

\begin{corollary}\label{H1_ER}
The map $\mathcal{E}R: \Gamma_c^\infty(M, E)\to \Gamma_c^\infty(M,E)$ extends to a bounded linear map  $\mathcal{E}R: H_s(M, E)\to H_s(M,E)$ for all $s>\frac{1}{2}$.
\end{corollary}

\begin{lemma}\label{pairing-Sobolev} The $L^2$-product $(\phi, \psi)= \int_\Sigma \langle \phi,\psi\rangle dv$ for $\phi, \psi\in \Gamma_c^\infty (\Sigma, E|_\Sigma)$ extends to a perfect pairing 
$H_s(\Sigma,E|_\Sigma) \times H_{-s}(\Sigma,E|_\Sigma)\to \mathbb{C} $ for all $s\in \mathbb{R}$.
\end{lemma}

\begin{proof}
 This is also proven in the same way as above -- by lifting the corresponding result from the Euclidean case \cite[Section I.3]{taylor_81}.
\end{proof}

The Trace Theorem now allows to extend the allowed domain for the spinors in the Equalities \eqref{Lich} and \eqref{L2-structure_mod_boundary}.

\begin{lem}\label{extequ}  For all $\phi,\psi\in H_1(M, \SS_M)$, Equalities \eqref{Lich} and  \eqref{L2-structure_mod_boundary} hold.
\end{lem}
\begin{proof} The proof is a more or less straightforward usage of the Trace Theorem \ref{trace_theorem} and the corresponding equalities on $\Gamma_c^\infty (M,\SS_M)$. Indeed, let $\phi_i$ be a sequence in $\Gamma_c^ \infty(M, \SS_M)$ with $\phi_i\to \phi$ in
$H_1(M, \SS_M)$. The Trace Theorem \ref{trace_theorem} gives $R\phi_i\to R\phi$ in
$H_{\frac{1}{2}}(\Sigma, \SS_M|_\Sigma)$ and, hence, $\widetilde{D}^{\Sigma} R\phi_i\to \widetilde{D}^{\Sigma}R\phi$ in
$H_{-\frac{1}{2}}(\Sigma, \SS_M|_\Sigma)$, cf. Remark \ref{rem_app}.iii. Clearly, 
$\Vert \phi_i-\phi\Vert_{H_1}\to 0$ and with \eqref{equ_H1D_easydir}, this implies 
$\Vert \phi_i-\phi\Vert_{D}\to 0$. Moreover, the bounded geometry of $(M, \Sigma)$ 
implies 
$$\int_{M}
{\s^M}|\phi_i|^2 dv \to \int_{M} {\s^M}|\phi|^2 dv,\ 
\int_{\Sigma} H|\phi_i|^2 ds \to \int_{\Sigma} H|\phi|^2 ds, \text{\ and}$$
$$\left|
\int_{M} \<\Omega\cdot\phi_i, \phi_i \> dv - \int_{M} \<\Omega\cdot\phi, \phi \>
dv\right| \leq  \left( \Vert \phi_i-\phi\Vert_{L^2}\Vert \phi\Vert_{L^2}+\Vert \phi_i\Vert_{L^2}\Vert \phi_i-\phi\Vert_{L^2}\right)\sup_M |\Omega|\to
0.$$ 
Note that due to the bounded geometry of $L$, $\sup_M |\Omega|$ is finite. It remains to consider the term $\int_{\Sigma} \langle R\phi,
\widetilde{D}^{\Sigma}R\phi\rangle ds$. First we note that due to the pairing in Lemma
\ref{pairing-Sobolev}, the Trace Theorem \ref{trace_theorem}, and $\widetilde{D}^{\Sigma}: H_{\frac{1}{2}}(\Sigma, \SS_M|_\Sigma)
\to H_{-\frac{1}{2}}(\Sigma, \SS_M|_\Sigma)$, this expression is finite for
all $\phi\in H_{1}(M, \SS_M)$. Abbreviating $R\phi$ by $\phi$,  we have
\begin{align*}
 | (\widetilde{D}^{\Sigma}\phi_i,\phi_i)_\Sigma - (\widetilde{D}^{\Sigma}\phi,\phi)_\Sigma| & \leq
 | (\widetilde{D}^{\Sigma}\phi_i, \phi - \phi_i)_\Sigma| +
|(\widetilde{D}^{\Sigma}\phi-\widetilde{D}^{\Sigma}\phi_i,\phi)_\Sigma|\\
 & \leq \Vert \widetilde{D}^{\Sigma}\phi_i\Vert_{H_{-\frac{1}{2}}} \Vert \phi -
\phi_i\Vert_{H_{\frac{1}{2}}} + \Vert
\widetilde{D}^{\Sigma}\phi-\widetilde{D}^{\Sigma}\phi_i\Vert_{H_{-\frac{1}{2}}} \Vert \phi
\Vert_{H_{\frac{1}{2}}},
 \end{align*}
 which gives the convergence of the last term. This proves Equality \eqref{Lich} for all $\phi\in H_1(M, \SS_M)$. 
Now, let $\phi_i,\psi_j$ be sequences in $\Gamma_c^\infty(M, \SS_M)$ with $\phi_i\to \phi$ and $\psi_j\to \psi$ in
$H_1(M, \SS_M)$. Then,
\begin{align*} |(D\psi_j,\phi_i)-(D\psi,\phi)|&\leq  |(D\psi_j,\phi_i)-(D\psi_j,\phi)|+|(D\psi_j,\phi)-(D\psi,\phi)|\\
&\leq  \Vert D\psi_j\Vert_{L^2} \Vert\phi_i-\phi\Vert_{L^2}+\Vert D(\psi_j-\psi)\Vert_{L^2} \Vert\phi\Vert_{L^2}.
\end{align*}
Using \eqref{equ_H1D_easydir} and that $\phi_i$ and $\psi_j$ are uniformly bounded in $H_1$, we get for a certain constant $C>0$ that \[ |(D\psi_j,\phi_i)-(D\psi,\phi)|\leq  C\Vert \phi_i-\phi\Vert_{L^2}+ C\Vert \psi_j-\psi\Vert_{H_1}\to 0.\]
Analogously, one obtains $(\psi_j, D\phi_i)\to (\psi, D\phi)$. Moreover, using again the Trace Theorem \ref{trace_theorem},
we get
\begin{align*} \left| \int_{\Sigma} \langle \nu\cdot R\psi_j, R\phi_i\rangle - \langle \nu\cdot R\psi_j, R\phi\rangle ds\right|&\leq \Vert R\psi_j\Vert_{L^2(\Sigma)} \Vert R(\phi_i-\phi)\Vert_{L^2(\Sigma)}\\
 &\leq C \Vert \psi_j\Vert_{H_1} \Vert \phi_i-\phi\Vert_{H_1}\to 0.
\end{align*}
In the same way, $ \left| \int_{\Sigma} \langle \nu\cdot R\psi_j, R\phi\rangle - 
\langle \nu\cdot R\psi, R\phi\rangle ds\right|\to 0$. Hence, \[ \left| \int_{\Sigma} \langle \nu\cdot R\psi_j, R\phi_i\rangle - \langle \nu\cdot R\psi, R\phi\rangle ds\right|\to 0.\]
This proves Equality \eqref{L2-structure_mod_boundary} for all $\phi, \psi\in H_1(M, \SS_M)$.
\end{proof}

\subsection{Extension and the graph norm}
\paragraph{\bf Spectral decomposition of the boundary}

Let $(M,\Sigma)$ be of bounded geometry. Then, $(\Sigma, g|_\Sigma)$ is complete and, thus, the  Dirac operator $D^\Sigma$ on $\SS_\Sigma$, and thus also $\widetilde{D}^{\Sigma}$ on $\SS_M|_\Sigma$, is self-adjoint.

Let $\{E_I\}_{I\subset \mathbb{R}}$ be the family of projector-valued measures
belonging to the self-adjoint operator $$\widetilde{D}^{\Sigma}: H_1(\Sigma,
\SS_M|_\Sigma) \subset L^2(\Sigma, \SS_M|_\Sigma)\to L^2(\Sigma,
\SS_M|_\Sigma).$$
We define for a connected (not necessarily bounded) interval $I\in \R$ the spectral projection
\begin{align*}
\pi_{I}: L^2(\Sigma, \SS_M|_{\Sigma})\to L^2(\Sigma,
\SS_M|_{\Sigma}),\ \phi\mapsto E_{I}\phi
\end{align*}
and the spaces
\[\Gamma^{\rm APS}_{I}=\{ \phi\in L^2(\Sigma, \SS_M|_{\Sigma})\ |\
\phi=\pi_{I} \phi \}.\]

Next we will show that for every $s\in \mathbb{R}$ the spectral projections extend to bounded linear maps from $H_s(\Sigma, \SS_M|_\Sigma)$ to itself: Firstly, we note that the spectral projections commute with $\widetilde{D}^{\Sigma}$. Moreover, since $(\Sigma, g|_\Sigma)$ has bounded geometry, the norm $\Vert \phi\Vert_{L^2}^2+ \Vert D^k \phi\Vert_{L^2}^2$ and the $H_k$-norm are equivalent on $\Gamma_c^\infty(\Sigma, \SS_M|_\Sigma)$ for $k\in \mathbb{N}_0$, cp. \cite[Lemma 3.1.6]{Ammha}.
Hence, $\pi_I$ restricts to a bounded linear map from $H_k(\Sigma, \SS_M|_\Sigma)$ to itself for $k\in \mathbb{N}_0$. Let now $k$ be a negative integer, $\phi\in \Gamma_c^\infty(\Sigma, \SS_M|_\Sigma)$ and $\psi\in H_{-k}(\Sigma, \SS_M|_\Sigma)$. Using that $\pi_I$ is symmetric w.r.t.   $L^2$-product on $(\Sigma, \SS_M|_\Sigma)$ and Lemma \ref{pairing-Sobolev} we get
\begin{align*}
 |(\pi_I \phi, \psi)_\Sigma|=|( \phi, \pi_I \psi)_\Sigma|\leq C\Vert \phi\Vert_{H_{-k}(\Sigma)}\Vert \pi_I \psi\Vert_{H_k(\Sigma)}\leq C'\Vert \phi\Vert_{H_{-k}(\Sigma)}\Vert \psi\Vert_{H_k(\Sigma)}.
\end{align*}
Thus, $\pi_I$ extends to a bounded linear map from $H_k(\Sigma, \SS_M|_\Sigma)$ to itself for all nonnegative integers $k$. Then by Riesz-Thorin Interpolation Theorem we get that $\pi_I: H_s(\Sigma, \SS_M|_\Sigma)\to H_s(\Sigma, \SS_M|_\Sigma)$ for all $s\in \mathbb{R}$.

We abbreviate $\pi_{>}=\pi_{(0,\infty)}$ and $\pi_{\leq}=\pi_{(-\infty, 0]}$. As in \cite[Section 5]{baer_ballmann_11}, we define for $\phi\in \Gamma_c^\infty(\Sigma, \SS_M|_\Sigma)$
\begin{align*}
 \Vert \phi\Vert_{\check{H}}^2&= \Vert \pi_{\leq }\phi\Vert_{H_\frac{1}{2}(\Sigma)}^2 + \Vert \pi_{>}\phi\Vert_{H_{-\frac{1}{2}}(\Sigma)}^2\ \text{\ and\ }
\Vert \phi\Vert_{\hat{H}}^2= \Vert \pi_{\leq }\phi\Vert_{H_{-\frac{1}{2}}(\Sigma)}^2 + \Vert \pi_{>}\phi\Vert_{H_{\frac{1}{2}}(\Sigma)}^2
\end{align*}
and the spaces \begin{align}\label{H_spaces}
\check{H}:=\overline{\Gamma_c^\infty(\Sigma, \SS_M|_\Sigma)}^{\Vert.\Vert_{\check{H}}}\ 
\text{\ and\ }\  \hat{H}:=\overline{\Gamma_c^\infty(\Sigma, \SS_M|_\Sigma)}^{\Vert.\Vert_{\hat{H}}}.\end{align}

\paragraph{\bf Local description of the graph norm on $(M,\Sigma)$.}

Let $(M,g)$ be a manifold with boundary $\Sigma$. Let $(U_\gamma, \kappa_\gamma, \xi_\gamma, h_\gamma)_\gamma$ be Fermi coordinates on  $(M,g)$ together with a synchronous trivialization $\xi_\gamma$ and a partition of unity $h_\gamma$. 

\begin{lemma} \label{lem_equiv_cutoff}
On  $\Gamma^\infty_c(M, \SS_M)$ the norms $\Vert.\Vert_D$ and $\left( \sum_\gamma \Vert h_\gamma .\Vert_D^2\right)^\frac{1}{2}$ are equivalent.
\end{lemma}

\begin{proof} All the constants $c_i$ involved here are positive.
 Let $\phi\in \Gamma_c^\infty(M,\SS_M)$. Since the cover $U_\gamma$ is uniformly locally finite the norms $\Vert.\Vert_{L^2}$ and $\left(\sum_\gamma \Vert h_\gamma . \Vert_{L^2}^2\right)^{\frac{1}{2}}$ are equivalent.
Thus, 
\begin{align*}
 \Vert D\phi\Vert_{L^2}^2& \leq c_1 \sum_\gamma \Vert h_\gamma D\phi\Vert_{L^2}^2=c_1 \sum_\gamma \Vert D(h_\gamma \phi)-\nabla h_\gamma\cdot \phi\Vert_{L^2}^2\\
&\leq c_2 \sum_\gamma (\Vert D(h_\gamma \phi)\Vert_{L^2}^2 + \Vert \nabla h_\gamma\cdot \phi\Vert_{L^2}^2)\leq c_3 \sum_\gamma (\Vert D(h_\gamma \phi)\Vert_{L^2}^2 + \Vert \phi|_{U_\gamma}\Vert_{L^2}^2)\\
&\leq c_3\sum_\gamma \Vert D(h_\gamma \phi)\Vert_{L^2}^2 + c_4 \Vert \phi\Vert_{L^2}^2
\end{align*}
where the end of the second line follows by Lemma \ref{part_un}, and the last inequality follows since the cover $U_\gamma$ is uniformly locally finite. Hence, $\Vert \phi\Vert_D^2\leq c_5 \sum_\gamma \Vert h_\gamma \phi\Vert_D^2$.

Conversely we get analogously
\begin{align*}
 \sum_\gamma \Vert D(h_\gamma \phi)\Vert_{L^2}^2&= \sum_\gamma \Vert h_\gamma D \phi+\nabla h_\gamma\cdot \phi\Vert_{L^2}^2\leq c_6  \Vert \phi\Vert_{D}^2.
\end{align*} \end{proof}

\begin{lemma}\label{local_inv}
 Let $(\Sigma, g|_\Sigma)$ be a manifold of bounded geometry. Then, there is an $\epsilon>0$ smaller than the injectivity radius of $\Sigma$ and a constant $c>0$ such that for all $x\in \Sigma$ and $\phi\in \Gamma^\infty_c(B_\epsilon(x)\subset N, \SS_N)$  we have $\Vert D^N\phi\Vert_{L^2}>c\Vert \phi\Vert_{L^2}$.
\end{lemma}

\begin{proof}  Let $\exp^{\Sigma}_x: B_\epsilon(0)\subset \R^n\to B_\epsilon(x)\subset \Sigma$ be the exponential map. Set $\tilde{g}:= (\exp^{\Sigma}_x)^* g|_{B_{\epsilon}(x)}$. We will compare the Dirac operator $D^{\tilde{g}}$ with $D^E$, \cite[Proposition 3.2]{AGHM}:
 \[ D^{\tilde{g}}\phi = D^E\phi + \sum_{ij} (b_i^j-\delta_i^j)\partial_i \cdot \nabla_{\partial_j}\phi +\frac{1}{4}\sum_{ijk} \tilde{\Gamma}_{ij}^k e_i\cdot e_j\cdot e_k\cdot \phi\]
 where $\phi$ is a smooth spinor over $ B_{\epsilon}(0)$, $\partial_i$ and $e_i$ form an orthonormal basis w.r.t. the Euclidean metric and $\tilde{g}$, respectively. Moreover, $e_i=b_i^j\partial_j$, $\nabla$ is the Levi-Civita connection w.r.t. the Euclidean metric, and $\tilde{\Gamma}_{ij}^k$ are the Christoffel symbols w.r.t. the metric $\tilde{g}$. 
 By \cite[(11)-(13) and below]{AGHM} $|b_i^j-\delta_i^j|\leq C r^2$ and $|\tilde{\Gamma}_{ij}^k|\leq Cr$ where $r$ is the Euclidean distance to the origin and $C$ can be chosen to only depend on the global curvature bounds of $g$.
 Moreover, note that there is a positive constant $C$ also depending only on the global curvature bounds of $g$ such that $C^{-1}\leq f\leq C$ where ${\rm dvol}_{\tilde{g}}= f{\rm dvol}_{g_E}$. 
 Thus, for $\epsilon$ small enough  we can estimate for all smooth spinors $\phi$ compactly supported in $B_\epsilon(0)$
 that 
 \begin{align*}
  \frac{\Vert D^{\tilde{g}} \phi\Vert_{L^2(\tilde{g})}^2 }{\Vert \phi\Vert_{L^2(\tilde{g})}^2}
  &\geq   C_1 \frac{\Vert D^E \phi\Vert_{L^2(g_E)}^2}{\Vert \phi\Vert_{L^2(g_E)}^2} - C_2\epsilon^2 \frac{\Vert \nabla \phi\Vert_{L^2(g_E)}^2}{\Vert \phi\Vert_{L^2(g_E)}^2} - C_3\epsilon\\
  &\geq   C_4 \frac{\Vert D^E \phi\Vert_{L^2(g_E)}^2}{\Vert \phi\Vert_{L^2(g_E)}^2}  - C_5\epsilon
   \end{align*}
where the last step uses the equivalence of the graph norm and the $H_1$-norm .
Let $A$ be such that $\Vert D^E \psi\Vert_{L^2(g_E)}^2 \geq A \Vert \psi\Vert_{L^2(g_E)}^2$ for smooth spinors compactly supported in $B_{\epsilon}(0)$. Then one can always choose $\epsilon$ small enough that $C_4A-C_5\epsilon\geq 2^{-1}C_4A=:c$
Thus, the same is true for $D^g$ on $B_\epsilon(x)\subset \Sigma$.
\end{proof}

Let $(\hat{M}, \hat{N}=\partial \hat{M})$ be manifold of bounded geometry with closed boundary.
Let $\mathcal{E}_{\rm BB}$ be an extension map as defined in \cite[(43)]{baer_ballmann_11}. 
Let $D$ and $D^{\hat{N}}$ be the Dirac operators on $\hat{M}$ and $\hat{N}$, respectively. By \cite[Lemma 6.1, 6.2, (41) and below]{baer_ballmann_11} we have for all $\phi\in \Gamma_c^\infty (\hat{M}, \SS_{\hat{M}}|_{\hat{N}})$
\begin{align}\Vert \mathcal{E}_{\rm BB}R\phi\Vert_D \leq C \Vert \phi\Vert_D\label{BB_Thm}.\end{align} Note that $C$ can be chosen to  depend only on curvature bounds of $(\hat{M}, \hat{N})$ including mean curvature, the injectivity radii of $\hat{M}$ and $\hat{N}$, respectively, and the spectral gap of $D^{\hat{N}}$.\\

We now come back to our pair $(M, N)$: Let $\epsilon, c>0$ be constants such that Lemma \ref{local_inv} is fulfilled. 
Let $(U_\gamma, \kappa_\gamma, \xi_\gamma, h_\gamma)$ be Fermi coordinates together with a subordinated partition of unity such that there are $x_\gamma\in \Sigma$ with  $U_\gamma\cap \Sigma \subset B_{\epsilon} (x_\gamma)$. Let $\hat{U}_\gamma$ be a manifold with closed boundary $\hat{U}'_\gamma:=\partial \hat{U}_\gamma$ such that $\tilde{U}_\gamma :=U_\gamma\cup (\cup_{\alpha; U_\alpha\cap U_\gamma\neq \varnothing} U_\alpha)$ can be isometrically embedded in $\hat{U}_\gamma$, $\tilde{U}_\gamma\cap \Sigma\subset \hat{U}'_\gamma$, such that the spectral gap of the Dirac operator on $\hat{U}'_\gamma$ is at least $[-c, c]$ and such that the curvature, mean curvature of the boundary and the injectivity radii are still uniformly bounded in $\gamma$. 

Define the map $\tilde{\mathcal{E}}: \Gamma_c^\infty(\Sigma, \SS_M|_\Sigma)\to \Gamma_c^\infty(M, \SS_M)$  by 
\[ \tilde{\mathcal{E}}\psi= \sum_{\gamma,\alpha;\  U_\gamma'\neq \varnothing, U_\gamma\cap U_\alpha\neq \varnothing} h_\alpha \mathcal{E}_{\rm BB} (h_\gamma|_{\Sigma} \psi) \]
where $h_\gamma|_{\Sigma} \phi$ is understood to be a spinor on $U_\gamma\cap N \subset \hat{U}'_\gamma$ and then $\mathcal{E}_{\rm BB} (h_\gamma|_{\Sigma} \psi)$ is a spinor on $\hat{U}_\gamma$. The only reason why $\sum_\alpha h_\alpha$ appears in the definition is to assure that each summand can be seen to live on $M$ and that $R\tilde{\mathcal{E}}=\Id$. Note that just using $h_\gamma$ in front of $\mathcal{E}_{\rm BB}$ would be enough to first requirement but not the second. 

\begin{proposition}\label{workaround_image} Using the notations from above, 
there is a positive constant $C$ such that for all $\phi\in\Gamma_c^\infty(M, \SS_M)$  
\begin{align*}
 \Vert \tilde{\mathcal{E}}R\phi\Vert_D\leq C\Vert\phi\Vert_D.
\end{align*}
\end{proposition}

\begin{proof} We abbreviate $h_\gamma':=h_\gamma|_\Sigma$. Using (in this order) the definition of $\tilde{\mathcal E}$, Lemma \ref{lem_equiv_cutoff} the uniform local finiteness of the cover $U_\gamma$, \eqref{BB_Thm}, and again Lemma \ref{lem_equiv_cutoff} we estimate 
\begin{align*}
 \Vert \tilde{\mathcal E}R\phi\Vert_D^2\leq& C_1 \left\Vert \sum_{\gamma, U'_\gamma\neq \varnothing}\mathcal{E}_{\rm BB} R(h_\gamma \phi)\right\Vert_D^2\leq C_2\sum_{\gamma, U'_\gamma\neq \varnothing} \Vert \mathcal{E}_{\rm BB} R(h_\gamma \phi)\Vert_D^2\\
\leq& 
C_3\sum_{\gamma, U'_\gamma\neq \varnothing} \Vert h_\gamma \phi\Vert_{D}^2\leq C \Vert \phi\Vert_D^2.
\end{align*}

\end{proof}

\section{Boundary value problems}\label{boundary_values}
The general theory of boundary value problems for elliptic differential operators of order one on complete 
manifolds with closed boundary can be found in \cite{baer_ballmann_11}. The aim of this section is to generalize a part of this theory to noncompact boundaries on manifolds
of bounded  geometry. We restrict to the part that gives existence of solutions of boundary value problems as in 
Theorem \ref{main}. The property needed to assure 
a solution to such a problem is the closedness of the range. For that we introduce a type of coercivity condition which 
in general can depend on the boundary values (that is not the case for closed boundaries).
Moreover, we restrict to the classical  $\Spinc$ Dirac operator.\\

In the first part, we first give some generalities on domains of the Dirac 
operator and introduce a coercivity condition that implies closed range of the Dirac operator. Then, 
 we extend the  trace map $R$ to the whole maximal domain of the Dirac operator 
and give some examples and properties of boundary conditions. In particular, we will introduce two boundary 
conditions $B_\pm$ which will be used to prove Theorem \ref{main} in Section \ref{proofmain}.
At the end, we give an existence result for boundary value problems in our context.\\

\paragraph{\bf General domains and closed range.} Let $D$ be the Dirac operator acting on $\Gamma_{cc}^\infty(M, \SS)$ on a manifold $M$ with boundary $\Sigma$. If we want to emphasize that $D$ acts on the domain $\Gamma_{cc}^\infty(M, \SS)$, we shortly write $D_{cc}$. We denote the graph norm of $D$ by 
\[ \Vert \phi\Vert_D^2=\Vert \phi\Vert_{L^2}^2+\Vert D \phi\Vert_{L^2}^2.\] 
By $D_{\mathrm{max}}:=(D_{cc})^*$ we denote the maximal extension of $D$. Here, $A^*$ denotes the adjoint operator of $A$ in the sense of functional analysis. Note that $H_1(M, \SS_M)\subset \dom\, D_{\mathrm{max}}$ and
 $$\dom\, D_{\mathrm{max}}=\{\phi\in L^2(M, \SS_M)\ |\ \exists \tilde{\phi}\in L^2(M, \SS_M) \forall \psi\in \Gamma_{cc}^\infty(M, \SS_M): (\tilde{\phi}, \psi)=(\phi, D\psi)\},$$
and together
with $\Vert. \Vert_D$, the space $\dom\, D_\mathrm{max}$ is a Hilbert space. Moreover, we denote by $D_{\mathrm{min}}:= (D_{cc})^{**} =\overline{D_{cc}}^{\Vert.\Vert_D}$ the minimal extension of $D$.  Here, $\overline{A}^{\Vert .\Vert_D}$ denotes the closure of the set $A$ w.r.t. the graph norm. Any closed linear subset of $\mathrm{dom}\, D_{\mathrm{max}}$ between $\mathrm{dom}\, D_{\mathrm{min}}$ and
 $\mathrm{dom}\, D_{\mathrm{max}}$ gives the domain of a closed extension of $D\colon \Gamma_{cc}^\infty (M, \SS_M) \to \Gamma_{cc}^\infty (M, \SS_M)$. Before examining those domains let us extend the trace map to $\dom\, D_{\rm max}$: \\
 
{\bf Extension of the trace map.} The Trace Theorem \ref{trace_theorem} extends the trace map 
\begin{eqnarray*}
 R: \Gamma_c^\infty(M,\SS_M) &\to& \Gamma_c^\infty(\Sigma,\SS_M|_\Sigma)\\
\phi&\mapsto& \phi|_\Sigma
\end{eqnarray*}
to a bounded map $R:H_1(M,\SS_M)\to H_{\frac{1}{2}}(\Sigma, \SS_M|_\Sigma)$.
Here, we will extend $R$ further to $\dom\, D_{\mathrm{max}}$. This will generalize the corresponding result \cite[Theorem 6.7.ii]{baer_ballmann_11} for closed boundaries to noncompact boundaries. Moreover, we give some auxiliary lemmata which are found in \cite{baer_ballmann_11} for closed boundaries. Some of the proofs and the order of obtaining them will be a little bit different since we do not use (and cannot use, cf. Example \ref{ex_bd_cond}.iv) the projection to the negative spectrum.Note that in this part we could use an arbitraray extension map as given by Theorem \ref{trace_theorem} and are not restricted to the explicit one defined via the eigenvalue decomposition of $\widetilde{D}^{\Sigma}$ on closed boundaries used in \cite{baer_ballmann_11}.
\begin{lem}\label{dense_in_graph_norm}
 The space $\Gamma_c^\infty(M,\SS_M)$ is dense in $\dom\, D_{\mathrm{max}}$ w.r.t. the graph norm.
\end{lem}
\begin{proof} For a closed boundary, this is done in \cite[Theorem 6.7.i]{baer_ballmann_11}. We use a different proof here. 
Let $\phi \in \dom\, D_{\mathrm{max}}$. Let $K_i$ be a compact exhaustion of $M$ that comes together with smooth  cut-off functions $\eta_i:M \to [0,1]$ such that $\eta_i=1$ on $K_i$, $\eta_i=0$ on $M\setminus K_{i+1}$ and $\max |d\eta_i|\leq \frac{2}{i}$. Then, 
$\phi_i=\eta_i\phi$ are compactly supported sections in $\dom\, D_{\max}$ fulfilling
\begin{eqnarray*}
 \Vert \phi_i-\phi\Vert_D^2 &=& \Vert \phi_i-\phi\Vert_{L^2}^2 + \Vert D\phi_i-D\phi\Vert_{L^2}^2\\ &\leq& 
\Vert (1-\eta_i)\phi\Vert_{L^2}^2 + \left( \Vert (1-\eta_i) D\phi\Vert_{L^2} +\frac{2}{i} \Vert \phi\Vert_{L^2} 
\right)^2\to 0.
\end{eqnarray*}
Each $\phi_i$  has now compact support in  $K_{i+1}$. Thus, there is a 
sequence $\phi_{ij}\in \Gamma_c^\infty (K_{i+1}, \SS_M)$ with $\phi_{ij}\to\phi_i$ in the graph 
norm on $K_{i+1}$. Choose $j=j(i)\geq i$ such that $\Vert \phi_{ij}-\phi_i\Vert_D\to 0$ as $i\to \infty$. Then, $\Vert \phi_{ij}-\phi\Vert_D \leq \Vert \phi_{ij}-\phi_i\Vert_D + \Vert \phi_i- \phi\Vert_D  \to 0$, too. Then
\begin{eqnarray*}
\Vert \eta_j\phi_{ij}-\phi_{ij}\Vert_D^2 &\leq& \Vert (1-\eta_j)\phi_{ij}\Vert_{L^2}^2 + (\Vert (1-\eta_j)D \phi_{ij}
\Vert_{L^2} + \Vert d\eta_j\cdot \phi_{ij}\Vert_{L^2}  )^2 
\\ &\leq& ( \Vert \phi_{ij}-\phi_i\Vert_{L^2} + \Vert (1-\eta_j)\eta_i\phi\Vert_{L^2})^2 +  
 \Big( \Vert D(\phi_{ij}-\phi_i)\Vert_{L^2}\\
 &&+\Vert(1-\eta_j)(\eta_i D\phi + d\eta_i \cdot \phi)\Vert_{L^2}  + \frac{2}{j} \Vert \phi_{ij}-\phi_i\Vert_{L^2} 
+ \frac{2}{j} \Vert \phi\Vert_{L^2}\Big)^2  \to 0
\end{eqnarray*}
for $i\to \infty$. Thus, we have a sequence $\hat{\phi}_i:=\eta_{j(i)}\phi_{ij(i)}\in \Gamma_c^ \infty(M, \SS_M)$ such that $\hat{\phi}_i\to \phi$ in the graph norm as $i\to \infty$.
\end{proof}

Note that the proof of Lemma \ref{dense_in_graph_norm} only uses the completeness of $M$ and not the bounded geometry.

\begin{thm}\label{extended-trace} 
 The trace map $R: \Gamma_{c}^\infty(M, \SS_M) \to \Gamma_{c}^\infty(\Sigma, \SS_M|_\Sigma)$ can be extended to a bounded operator 
 $$R: \dom\, D_{\mathrm{max}} \to H_{-\frac{1}{2}} (\Sigma, \SS_M|_\Sigma).$$
 \end{thm}

\begin{proof} Let $\phi\in \Gamma_c^\infty (M, \SS_M)$ and $\psi\in  H_{\frac{1}{2}} (\Sigma, \SS_M|_\Sigma)$. Then by Theorem \ref{trace_theorem}, the spinor 
$\mathcal{E}\psi\in H_1(M, \SS_M)$. Thus, we can use Lemma \ref{extequ}, \eqref{equ_H1D_easydir}, and Theorem \ref{trace_theorem}
 to obtain
 \begin{align*}
 |( \phi|_\Sigma, \nu\cdot \psi)_\Sigma|=&|(D\phi, \mathcal{E} (\nu\cdot\psi))-(\phi, D\mathcal{E}(\nu\cdot\psi))|\\
 \leq &\Vert D\phi\Vert_{L^2}\Vert \mathcal{E}(\nu\cdot\psi)\Vert_{L^2}+ \Vert \phi\Vert_{L^2}\Vert D \mathcal{E}(\nu\cdot\psi)\Vert_{L^2}\\
 \leq &2\Vert \phi\Vert_{D}\Vert \mathcal{E}(\nu\cdot\psi)\Vert_{D} \leq C \Vert \phi\Vert_{D}\Vert \mathcal{E}(\nu\cdot\psi)\Vert_{H_1}
 \leq  C' \Vert \phi\Vert_{D}\Vert \nu\cdot\psi\Vert_{H_{\frac{1}{2}}(\Sigma)}.
\end{align*}
Together with Lemma \ref{pairing-Sobolev}, this implies
\[ \Vert \phi|_\Sigma\Vert_{H_{-\frac{1}{2}}(\Sigma)}\leq C' \Vert \phi\Vert_{D}.\]
Since $\Gamma_c^\infty (M, \SS_M)$ is dense in $\dom\, D_{\mathrm{max}}$ w.r.t. the graph norm, cf. Lemma \ref{dense_in_graph_norm}, the claim follows.
\end{proof}

\begin{remark}
 Note that $R$ is not surjective here. For closed boundaries the image was specified in 
\cite[Theorems 1.7 and 6.7.ii]{baer_ballmann_11}. For noncompact 
boundaries the image will be further considered in Lemma \ref{R_banach} and below.
\end{remark}

\begin{lem}\label{againextequ}  Equality \eqref{L2-structure_mod_boundary} holds for all 
$\phi\in \dom\, D_{\mathrm{max}}$ and $\psi\in H_1(M,\SS_M)$.
\end{lem}
\begin{proof}
The proof is done as the one of Lemma \ref{extequ} starting with $\psi_j, \phi_i\in \Gamma_c^\infty(M, \SS_M)$ where $\psi_j\to \psi$ in $H_1$ and $\phi_i\to \phi$ in the graph norm of $D$ and using the (extended) Trace Theorem \ref{extended-trace}. The only difference is seen in the estimate of the boundary integrals which now read e.g.
\begin{align*} \left| \int_{\Sigma} \langle \nu\cdot R\psi_j, R\phi_i - R\phi\rangle ds\right|&\leq \Vert R\psi_j\Vert_{H_{\frac{1}{2}}(\Sigma)} \Vert R(\phi_i-\phi)\Vert_{H_{-\frac{1}{2}}(\Sigma)}
\leq C\Vert \psi_j\Vert_{H_1} \Vert \phi_i-\phi\Vert_{D}\to 0
\end{align*}
where the last inequality uses both versions of the Trace Theorem \ref{trace_theorem} and \ref{extended-trace}.
\end{proof}
The next Lemma gives a full description of $\dom\, D_{\mathrm{min}}$:
\begin{lem}\label{norm_equ_0}  The $H_1$-norm and the graph norm $\Vert . \Vert_D$ are equivalent on
 $$\{\phi\in \dom\, D_\mathrm{max}\ |\ R\phi=0\}.$$ In particular, 
\begin{eqnarray*}
\dom\, D_{\mathrm{min}}=\overline{\Gamma_{cc}^\infty(M, \SS_M)}^{\Vert.\Vert_D}=\overline{\Gamma_{cc}^\infty(M, \SS_M)}^{\Vert.\Vert_{H_1}}&=&\{\phi\in \dom\, D_{\mathrm{max}}\ |\ R\phi=0\}\\ &=&\{\phi\in H_1(M, \SS_M)\ |\ R\phi=0\}.
\end{eqnarray*}
\end{lem}
\begin{proof}Firstly we show the equivalence on $\{\psi\in \Gamma_{c}^\infty(M, \SS_M)\ |\ R\psi=0\}$: Let $\phi$ be in this set. Then by \eqref{Lich} we have
 \[ \Vert \phi\Vert_{H_1}^2=\Vert \phi\Vert_{L^2}^2+\Vert D\phi\Vert_{L^2}^2 -\int_{M} \frac{\s^M}{4}|\phi|^2 dv - \int_{M} \frac {\i}{2}
\<\Omega\cdot\phi, \phi \> dv\leq C \Vert \phi\Vert_D^2, \]
where we used that $M$ and $L$ are of bounded geometry and, hence, $|\s^M|$ and $|\Omega|$ are uniformly bounded on all of $M$. The reverse inequality was seen in \eqref{equ_H1D_easydir}. 

From the definition of $\dom\, D_{\mathrm{min}}$ and the equivalence of the norms from above, we already have  $\dom\, D_{\mathrm{min}}=\overline{\Gamma_{cc}^\infty}^{\Vert.\Vert_D}=\overline{\Gamma_{cc}^\infty}^{\Vert.\Vert_{H_1}}$. From the Trace Theorem \ref{extended-trace}, we get 
$$\overline{\Gamma_{cc}^\infty}^{\Vert.\Vert_D}\subset \{\phi\in \dom\, D_{\mathrm{max}}\ |\ R\phi=0\}.$$ Next we want to show that $D\colon \{\phi\in \dom\, D_{\mathrm{max}}\ |\ R\phi=0\} \to L^2(M, \SS_M)$ already 
equals $D_{\mathrm{min}}$. First we note that by the Trace Theorem \ref{extended-trace}, $D$ is a closed extension of $D_{cc}$. 
Hence, it suffices to show that $D^*=D_{\mathrm{max}}$. By definition, we have 
\[\dom\, D^*=\{ \theta \in L^2(M,\SS_M)\ |\ \exists \chi\in L^2(M, \SS_M)\, \forall \psi\in \dom\, D_{\mathrm{max}}, R\psi=0: (\theta, D\psi)=(\chi,\psi)\}.\]
Let $\theta \in \dom\, D_{\mathrm{max}}$. By Lemma \ref{dense_in_graph_norm},  there exists 
a sequence $\theta_i\in \Gamma_c^\infty(M,\SS_M)$ with $\theta_i\to \theta$ in the graph norm. Hence, for all 
$\psi\in \dom\, D_{\mathrm{max}}$ with $R\psi=0$ we have $(\theta, D\psi)=\lim_{i\to \infty}(\theta_i,D\psi)$. Then by Lemma 
\ref{againextequ} and $R\psi=0$, 
we obtain $$(\theta, D\psi)=\lim_{i\to \infty}(D\theta_i,\psi)=(D\theta,\psi)$$ 
which implies that $\theta\in \dom\, D^*$. Thus, $D^*=D_{\mathrm{max}}$ and $D=D_{\mathrm{min}}$. Together with
 \[\dom\, D_{\mathrm{min}}=\overline{\Gamma_{cc}^\infty}^{\Vert.\Vert_{H_1}}\subset  \{\phi\in H_1(M,\SS_M)\ |\ R\phi=0\}\subset \{\phi\in \dom\, D_{\mathrm{max}}\ |\ R\phi=0\}=\dom\, D_{\mathrm{min}},\] the rest of the Lemma follows. 
\end{proof}

Now we can describe $H_1$ in terms of its image under the trace map.

\begin{lem}\label{H1_dommax}
We have $H_1(M, \SS_M)=\{\phi\in \dom\, D_{\mathrm{max}}\ |\ R\phi\in H_{\frac{1}{2}}(\Sigma, \SS_M|_\Sigma)\}$.
\end{lem}

\begin{proof}
The inclusion '$\subset$' is clear from the Trace Theorem \ref{trace_theorem}. It remains to prove '$\supset$': 
Let $\phi\in \dom\, D_{\mathrm{max}}$ with $R\phi\in H_\frac{1}{2} (\Sigma, \SS_M|_\Sigma)$. Then Theorem \ref{trace_theorem} 
   implies that 
 $\psi:=\mathcal{E}R\phi\in H_1 (M, \SS_M)$. Thus, $\phi-\psi\in \dom\, D_{\mathrm{max}}$ and $R(\phi-\psi)=0$. But
 due to Lemma \ref{norm_equ_0}, $\phi-\psi\in H_1(M, \SS_M)$ and, hence, $\phi\in H_1(M, \SS_M)$.
\end{proof}

In Proposition \ref{workaround_image} we have shown that there is a linear map $\tilde{\mathcal E}$ such that $\tilde{\mathcal{E}}R: \Gamma_c^\infty(M,\SS_M)\to \Gamma_c^\infty(M,\SS_M)$ fulfills for all $\phi\in \Gamma_c^\infty(M, \SS_M)$
\begin{align} \Vert \tilde{\mathcal{E}} R\phi\Vert_D^2\leq C\Vert \phi\Vert_D^2.\label{super_ext}\end{align}
Thus, $\tilde{\mathcal{E}}R$ extends uniquely to a bounded linear map 
\begin{align} 
 \tilde{\mathcal{E}}R: \dom\, D_{\rm max} \to \dom\, D_{\rm max} \label{ER_ext}.
\end{align}

Note that $\tilde{\mathcal{E}}|_{H_\frac{1}{2}}$ is an extension map in the sense of Theorem~\ref{trace_theorem} as can be seen in the following: 
 Let $\psi \in H_\frac{1}{2} (\Sigma, \SS_M|_\Sigma)$. By Lemma \ref{H1_dommax} there is a $\phi \in H_1(M, \SS_M)$ with $R\phi=\psi$. Thus, by Lemma \ref{norm_equ_0} $\tilde{\mathcal{E}}\psi-\phi\in \dom\, D_{\rm min}\subset H_1(M, \SS_M)$. In particular,   $\tilde{\mathcal{E}}|_{H_{\frac{1}{2}}}: H_\frac{1}{2} (\Sigma, \SS_M|_\Sigma)\to H_1(M, \SS_M)$.

From now we choose any extension map $\mathcal{E}$ fulfilling \eqref{super_ext}. Obviously, all those maps  lead to equivalent norms $\Vert \mathcal{E}R.\Vert_D$.

\begin{conjecture}
 Every extension map in the sense of Theorem \ref{trace_theorem} fulfills \eqref{super_ext} with an appropriate constant $C$.
\end{conjecture}

On $R(\dom\, D_{\mathrm{max}})$,  we set $$\Vert \psi\Vert_{\check{R}} := \Vert \mathcal{E}R\phi\Vert_D $$ 
where $R\phi=\psi$. By Theorem \ref{workaround_image} and \eqref{ER_ext}, this is  well defined. 

\begin{lemma}\label{R_banach} The space $\check{R}:=(R(\dom\, D_{\mathrm{max}}), \Vert. \Vert_{\check{R}})$ is a Hilbert space with $\check{R}= \overline{\Gamma_c^\infty(\Sigma, \SS_M|_\Sigma)}^{\Vert. \Vert_{\check{R}}}$.   
\end{lemma}

\begin{proof} From the definition of $\Vert. \Vert_{\check{R}}$, the linearity of the maps $\mathcal{E}$ and $R$, and the fact that $(\dom\, D_{\mathrm{max}}, \Vert. \Vert_D)$ is a Hilbert space, we get immediately that  
$\Vert .\Vert_{\check{R}}$ is a norm on $R(\dom\, D_{\mathrm{max}})$. Moreover, $\Vert.\Vert_{\check{R}}$ comes from a scalar product $(\phi,\psi)_{\check{R}}:=(\mathcal{E}\phi,\mathcal{E}\psi)_D\colon=(\mathcal{E}\phi,\mathcal{E}\psi) + (D\mathcal{E}\phi,D\mathcal{E}\psi)$. In order to show that $\check{R}$ is a Hilbert space it remains to show completeness: For that we consider a Cauchy 
sequence $\psi_i$ in $\check{R}$. Then, there is a sequence $\phi_i\in \dom\, D_\mathrm{max}$ with
 $R\phi_i=\psi_i$. With the definition of the $\check{R}$-norm,
 we get that $\mathcal{E}R\phi_i$ is a Cauchy sequence in  $(\dom\, D_{\mathrm{max}}, \Vert.\Vert_D)$ and, 
hence, there is a $\phi\in \dom\, D_{\mathrm{max}}$ with $\mathcal{E}R\phi_i\to \phi$ w.r.t. the graph norm. 
By Theorem \ref{workaround_image},  we get 
$$\Vert\mathcal{E}R(\phi_i-\phi)\Vert_D=\Vert\mathcal{E}R(\mathcal{E}R\phi_i-\phi)\Vert_D\leq C\Vert \mathcal{E}R\phi_i -\phi\Vert_D\to 0.$$ Thus, $\mathcal{E}R\phi=\phi$ and $\Vert \psi_i-R\phi\Vert_{\check{R}}= \Vert \mathcal{E} (R\phi_i-R\phi)\Vert_D\to 0$. Hence, $\psi_i\to \psi$ in the $\check{R}$-norm.

Clearly, $\overline{\Gamma_c^\infty(\Sigma, \SS_M|_\Sigma)}^{\Vert. \Vert_{\check{R}}}\subset R(\dom\, D_{\rm max})$. Let now $\psi\in R(\dom\, D_{\rm max})$. Then, there is a $\phi \in \dom\, D_{\rm max}$ with $R\phi=\psi$. By Lemma \ref{dense_in_graph_norm} there is a sequence $\phi_i\in \Gamma_c^\infty(M, \SS_M)$ with $\Vert \phi_i-\phi\Vert_D\to 0$ as $i\to \infty$. Thus, by Theorem \ref{workaround_image} the sequence $\psi_i:= R\phi_i\in \Gamma_c^\infty(\Sigma, \SS_M|_\Sigma)$ converges to $\psi$ in the $\check{R}$-norm.
\end{proof}

\begin{remark}\hfill\label{rem_norms}\\
\textbf{(i)} The proof of Proposition \ref{workaround_image} and \cite[Lemma 6.1]{baer_ballmann_11} implies 
\[\Vert \tilde{\mathcal E} R\phi\Vert_D^2\leq C' \sum_{\gamma, \hat{U}'_\gamma\neq \varnothing} \Vert R(h_\gamma \phi)\Vert_{\check{H}(\hat{U}'_\gamma)}^2 =: C' \Vert R\phi\Vert^2_{\check{H}_\gamma}.\] On the other hand, by  \cite[Lemma 6.2, (41) and below]{baer_ballmann_11} $\Vert R(h_\gamma \phi)\Vert_{\check{H}(\hat{U}'_\gamma)}^2\leq C \Vert h_\gamma \phi\Vert_D^2$ where $C$ again only depends on the curvature bounds of $(M,\Sigma)$ and the spectral gap $c$ on $\hat{U}'_\gamma$. Thus, together with Lemma \ref{lem_equiv_cutoff} the norms $\Vert .\Vert_{\check{R}}$ and $\Vert.\Vert_{\check{H}_\gamma}$ are equivalent.\\
\textbf{(ii)} Using $(i)$ and \cite[Lemma 6.3]{baer_ballmann_11} we see 
\begin{align*}
 \Vert \tilde{\mathcal E} (\nu\cdot R\phi)\Vert_D^2\leq C' \sum_{\gamma, U'_\gamma\neq \varnothing} \Vert \nu \cdot R(h_\gamma \phi)\Vert_{\check{H}(\hat{U}'_\gamma)}^2 = C' \sum_{\gamma, U'_\gamma\neq \varnothing} \Vert R(h_\gamma \phi)\Vert_{\hat{H}(\hat{U}'_\gamma)}^2=: \Vert R\phi\Vert_{\hat{H}_\gamma}^2.
\end{align*}
Together with \cite[Lemma 6.1]{baer_ballmann_11} we obtain for all $\phi\in \Gamma_c^\infty(M,\SS_M)$
\begin{align*}
 \Vert \tilde{\mathcal E} (\nu\cdot R\phi)\Vert_D^2\leq C \Vert \phi\Vert_D^2
\end{align*}
and, thus, $\Vert \psi\Vert_{\hat{R}} := \Vert \mathcal{E}(\nu\cdot R\phi)\Vert_D$ also gives rise to a norm on $R(\dom\, D_{\mathrm{max}})$. Moreover, the analogous statement of Lemma \ref{R_banach} holds for $\hat{R}:=(R(\dom\, D_{\mathrm{max}}), \Vert. \Vert_{\hat{R}})$, and we have $\Vert \psi\Vert_{\check{R}}=\Vert \nu\cdot \psi\Vert_{\hat{R}}$.
In particular, we get as in (i) that the norms $\Vert \tilde{\mathcal E} (\nu\cdot .)\Vert_D$ and $\Vert .\Vert_{\hat{H}_\gamma}$ are equivalent.
\end{remark}

\begin{remark} 
Note that by Theorem \ref{extended-trace}  and \ref{H1_dommax} \[H_{\frac{1}{2}}(\Sigma, \SS_M|_\Sigma)\subset (R(\dom\, D_{\mathrm{max}}), \Vert.\Vert_{\check{R}\, (\text{resp.\ } \hat{R})}) \subset H_{-\frac{1}{2}}(\Sigma, \SS_M|_\Sigma).\]
\end{remark}

Moreover, the perfect pairing of $\hat{H}_\gamma$ and $\check{H}_\gamma$, induced by the pairing of $H_{\frac{1}{2}}$ and  $H_{-\frac{1}{2}}$, gives immediately

\begin{lemma}\label{pair_R}
 The $L^2$-product on $\Gamma_c^\infty(\Sigma, \SS_M|_\Sigma)$ extends uniquely to a perfect pairing $\check{R}\times \hat{R} \to \mC$.
 \end{lemma}

Up to now we have seen that the $\check{R}$-norm is equivalent to the norm $\Vert.\Vert_{\check{H}_\gamma}$, cp. Remark \ref{rem_norms}.i where the second norm comes with an appropriate trivialization of the manifold near the boundary, see before Proposition \ref{workaround_image}. But we also think that as in the closed case there should be a 'more intrinsic' equivalent norm:

\begin{conjecture}
 The $\check{R}$-norm on $R(\dom\, D_{\rm max})$ is equivalent to the $\check{H}$-norm as defined in \eqref{H_spaces}. Moreover, $\check{H}=R(\dom\, D_{\rm max})$ as vector spaces.
\end{conjecture}

{\bf Boundary conditions.}
In this part, we show that each closed extension of $D_{cc}$ can be realized by a closed linear subset of $\check{R}$, and we give some examples.

 \begin{lem}\label{Bmax}
Let $D$ be a closed extension  of $D_{cc}$ with $B:=R(\dom\, D)\subset H_{-\frac{1}{2}} (\Sigma, \SS_M|_\Sigma)$. Then, 
 its domain $\dom\, D$ equals $\dom\, D_B\colon=\{ \phi\in \dom\, D_{\mathrm{max}}\ |\ R\phi\in B\}$, and $B$ is a closed linear subset of $\check{ R}$. Conversely, for every closed linear subset $B\subset \check{ R}$ the operator $D_B\colon \dom\, D_B \to L^2(M,\SS_M)$ is a closed extension of $D_{cc}$.
 \end{lem} 

 Due to this Lemma, a closed subspace $B$ of $\check{ R}$ is called {\it boundary condition}. 
 
 \begin{proof}
  Let $D$ be a closed extension of $D_{cc}$ with domain $\dom\, D$ and $B:= R(\dom\, D)$. Clearly, $\dom\, D\subset \dom\, D_B$. 
 We have to show that also the converse is true: Let $\phi \in \dom\, D_B$. Then,  there exists
 $\psi \in \dom\, D$  with $R \phi= R\psi$. By Lemma \ref{norm_equ_0}, 
 $\phi-\psi\in \dom\, D_{\mathrm{min}}\subset \dom\, D$ and, hence, $\phi\in \dom\, D$. This implies that $\dom\, D=\dom\, D_B$. 
Moreover, from \eqref{ER_ext} and the definition of the $\check{R}$-norm the maps $R: \dom\, D_{\mathrm{max}} \to \check{ R}$ and $\mathcal{E}: \check{ R}\to \dom\, D_{\mathrm{max}}$ are continuous. Hence, if $\dom\, D$ is closed in $\dom\, D_{\mathrm{max}}$, the set $B=\mathcal{E}^{-1}(\dom\, D)$ is closed in $R(\dom\, D_{\mathrm{max}})$. Conversely, if $B$ is closed in $\check{ R}$, 
$\dom\, D=R^{-1}(B)$ is closed in $\dom\, D_{\mathrm{max}}$.
\end{proof}

 \begin{lem}\label{equiv_H1_D}
  Let  $B$ be a boundary condition such that $B\subset H_{\frac{1}{2}}(\Sigma, \SS_M|_\Sigma)$. Then, the $H_1$-norm and the graph norm $\Vert.\Vert_D$ are equivalent on $\dom\, D_B$.
 \end{lem}
 
 \begin{proof}
  Since $B$ is a boundary condition, $\dom\, D_B$ is closed in $(\dom\, D_{\mathrm{max}}, \Vert. \Vert_D)$. Moreover, by $B\subset H_{\frac{1}{2}}(\Sigma, \SS_M|_\Sigma)$, Lemma \ref{H1_dommax} and \eqref{equ_H1D_easydir},  $\dom\, D_B$ is closed in $(H_1(M, \SS_M), \Vert.\Vert_{H_1})$. Thus, $(\dom\, D_B, \Vert.\Vert_D)$ and $(\dom\, D_B, \Vert.\Vert_{H_1})$ are both Hilbert spaces. By \eqref{equ_H1D_easydir} the identity map $\Id\colon (\dom\, D_B, \Vert.\Vert_{H_1})\to (\dom\, D_B, \Vert.\Vert_D)$ is a bijective bounded linear map. From the bounded inverse theorem we know that also the inverse is bounded. Hence, the $H_1$- and the graph norm are equivalent on $\dom\, D_B$. 
 \end{proof}

\begin{remark}\label{comparebb}  
The definition of $\dom\, D_B$ in \cite[Section 7]{baer_ballmann_11} uses $H_1^D\colon=\overline{\Gamma_c^\infty(M,\SS_M)}^{\Vert.\Vert_{H_1^D}}$ instead of $H_1$ where the $H_1^D$-norm
is given by $$\Vert \phi\Vert_{H_1^D}^2=\Vert \chi \phi\Vert_{H_1}^2 + \Vert \phi\Vert_{L^2}^2 + \Vert D\phi\Vert_{L^2}^2.$$
Here $\chi$ denotes an appropriate cut-off function such that $\chi\phi$ only lives on a small collar of the boundary. Since we work with the classical Dirac operator on $\Spinc$ manifolds and assume $(M,\Sigma)$ and $L$ being of 
bounded geometry, the $H_1$- and the $H_1^D$-norm coincide. Ch. B\"ar and W. Ballmann consider a more general situation where it suffices that
 $M$ is only complete but not necessarily of bounded geometry. Then the $H_1^D$-norm is needed. We could also switch 
to this more general setup when dropping the condition (i) and (iii) in the Definition \ref{bdd_geo} while still 
assuming that $(\Sigma, g|_\Sigma)$ is of bounded geometry and that the curvature tensor and its derivatives
 are bounded on $U_\Sigma$. For that situation, we would also obtain Theorem \ref{main}. But
 in order to simplify 
notation we stick to the bounded geometry of $(M,\Sigma)$.
\end{remark}

\begin{example}\label{ex_bd_cond}
\begin{itemize}
 \item[(i)] {\bf Minimal and maximal extension.} $B={0}$ gives the minimal extension $D_{B=0}=D_{\mathrm{min}}$, cf. Lemma \ref{norm_equ_0}. The maximal extension is obtained with $B=R(\dom\, D_\mathrm{max})$.
 \item[(ii)] $D_{B=H_{\frac{1}{2}}}\colon H_1(M,\SS_M)\to L^2(M,\SS_M)$ is an extension of $D_{cc}$ but not closed (if the boundary is nonempty): Since  $\Gamma_c^\infty(M,\SS_M)\subset H_1$ and $\Gamma_c^\infty(M,\SS_M)$ dense in $\dom\, D_{\mathrm{max}}$, the closure of ${D}_{B=H_{\frac{1}{2}}}$ is $D_{\mathrm{max}}$.
\item[(iii)]  \cite[Section 6]{HMZ02}
Let $P_\pm: L^2 (\Sigma, \SS_M|_\Sigma) \to L^2 (\Sigma, \SS_M|_\Sigma), \ \phi\mapsto \frac{1}{2}(\phi  \pm \i \nu\cdot \phi)$ and
 \[D_\pm\colon \dom\, D_\pm:=\{\phi\in \dom\, D_{\mathrm{max}}\ |\  P_\pm R\phi=0\}\to L^2(M,\SS_M).\] 
 In Section \ref{boundcond_B+-},  we will show that 
$D_\pm$ is a closed extension and that $D_\pm=D_{B_\pm}$ where 
\[B_\pm=\{\phi\in H_\frac{1}{2}(\Sigma,\SS_M|_\Sigma)\ | \ P_\pm\phi=0\}. \]

Each $\phi$ decomposes uniquely into $\phi=P_+ \phi +P_-\phi$, and if $\phi\in H_\frac{1}{2}(\Sigma, \SS_M|_\Sigma)$, then
$P_\pm\phi\in H_\frac{1}{2}(\Sigma, \SS_M|_\Sigma)$, too. This assures that the $B_\pm$'s are honestly larger than the trivial boundary condition $B=\{0\}$. More properties of this boundary condition can be found in Section \ref{boundcond_B+-}. 

 \item[(iv)] {\bf APS boundary conditions.} 
An obvious way to generalize the APS boundary conditions for a closed boundary to our situation is given by the following:
Let $(M,\Sigma)$ be of bounded geometry. We use the notations introduced in Section \ref{trace_theorem}.

We set  $B^{\mathrm{APS}}_{\geq a}=R(\dom\, D_\mathrm{max})\cap  \Gamma_{[a,\infty)}^{\mathrm{APS}}$ and $B^{\mathrm{APS}}_{< a}=R(\dom\, D_\mathrm{max})\cap  \Gamma_{ (-\infty, a]}^{\rm APS}$, respectively. In the same ways, let $B^{\mathrm{APS}}_{\leq a}$ and $B^{\mathrm{APS}}_{> a}$ be defined. If a neighbourhood of $a$ is in the spectrum of $D^\Sigma$,  $B^{\mathrm{APS}}_{< a}$ and $ B^{\mathrm{APS}}_{> a}$ won't be closed. 
We conjecture that for $(M,\Sigma)$ of bounded geometry the sets $B^{\mathrm{APS}}_{\geq a}$ and $B^{\mathrm{APS}}_{\leq a}$ define boundary conditions. But actually we don't know.
\end{itemize}
\end{example}

{\bf Boundary value problems.}  In this part we want to prove Theorem \ref{intro-bvp}. For that we need to define first the notion coercivity at infinity: 

\begin{definition}\label{coer}
A closed linear operator $D\colon \dom\, D\subset L^2(M,\SS_M)\to L^2(M,\SS_M)$ is said to be $(\dom\, D)$-coercive at infinity if there is a $c>0$ such that 
\[ \forall \phi\in \dom\, D\cap \left( \ker D\right)^\perp:\ \Vert D\phi\Vert_{L^2} \geq c\Vert \phi\Vert_{L^2}\]
where $\!\!\phantom{,}^\perp$ denotes the orthogonal complement in $L^2$.
\end{definition}

Note that in case that $D$ is the Dirac operator on a complete manifold without boundary, coercitivity at infinity follows immediately if $0$ is not the essential spectrum. Conversely if the Dirac operator is coercive at infinity then either $0$ is not in the essential spectrum or the kernel is infinite-dimensional. For manifolds with boundary, $D$ is in general no longer self-adjoint. Thus, the spectrum is in general complex and this translation to the essential spectrum is not possible.

In Section \ref{section_coer}, we will compare this coercivity condition with the originally one used in \cite[Defintion 8.2]{baer_ballmann_11} for closed boundaries. But first,  we will see how this condition forces the range of the operator to be closed which is crucial in order to apply the Closed Range Theorem \ref{closed_range_theorem} and  show existence of preimages for linear operator as we will need in Theorem \ref{intro-bvp}.

\begin{lemma} \label{coercive_closed}
If the closed linear operator $D\colon \dom\, D\subset L^2(M,\SS_M)\to L^2(M,\SS_M)$ is $(\dom\, D)$-coercive at infinity, then the range is closed.
\end{lemma}
\begin{proof}
Let $\phi_i$ be a sequence in $\dom\, D$ with $D\phi_i\to \psi$ in $L^2$. We have to show that $\psi$ is in the image of $D$.
W.l.o.g. we can assume that $\phi_i \perp \ker D$. Then  $(\dom\, D)$-coercivity at infinity gives that $\phi_i$ is bounded in $L^2$ and, thus, also in the
graph norm of $D$. Thus, $\phi_i\to \phi$ weakly in $\Vert. \Vert_D$. Let $\eta\in \dom\, D^*$. Then, $(D\phi,\eta)=\lim_{i\to\infty} (D\phi_i, \eta)=\lim_{i\to\infty} (\phi_i, D^*\eta)= (\phi, D^*\eta)$. Thus, $\phi\in \dom\, D$ and closedness of $\dom\, D$ then implies 
that $D\phi=\psi$.
\end{proof}

We are now ready to prove

\begin{reptheorem}{intro-bvp} Let $B$ be a boundary condition, and let the Dirac operator 
$$D_B\colon \dom\, D_B\subset L^2(M, \SS_M) \to L^2(M, \SS_M)$$ be $B$-coercive at infinity. Let $P_{B}\colon R(\dom\, D_\mathrm{max})\to B$ be a projection. 
Then, for all $\psi\in L^2(M, \SS_M)$ and $\tilde{\rho}\in \dom\, D_\mathrm{max}$ where $\psi-D\tilde{\rho}\in (\ker\, (D_{B})^*)^\perp$, the boundary value problem
$$\left\{
\begin{array}{rll}
 D\phi &=\psi  & \text{on\ } M,\\
(\Id - P_{B})R\phi& = (\Id - P_{B}) R\tilde{\rho}&\text{on\ } \Sigma
\end{array}
\right.
$$
has a solution $\phi\in \dom\, D_\mathrm{max}$ that is unique up to elements of the kernel $\ker\, D_B$.
\end{reptheorem}

Projection only means here that $P_B$ is linear and $P_B|_B=\Id$.
\begin{proof}
Since $D$ is $B$-coercive at infinity, its range is closed by Lemma \ref{coercive_closed}. Thus, 
due to the Closed Range Theorem \ref{closed_range_theorem}, the spinor $\psi-D\tilde{\rho}\in \mathrm{ran}\, D_B$. Hence, there exists $\hat{\phi}\in \dom\, D_B$ with $D\hat{\phi}=
\psi-D\tilde{\rho}$. Setting $\phi=\hat{\phi}+\tilde{\rho}$, we get $\phi\in \dom\, D_\mathrm{max}$, $D\phi=\psi$, and $(\Id - P_{B})R\phi= (\Id - P_{B})R\hat{\phi}+ (\Id - P_{B})R\tilde{\rho}= (\Id - P_{B})R\tilde{\rho}$.  
\end{proof}
\begin{corollary} \label{bvp-H1} Let $B$ be a boundary condition such that $B\subset H_{\frac{1}{2}}(\Sigma, \SS_M|_\Sigma)$. We 
assume that the Dirac operator $D\colon \dom\, D_B\subset L^2(M, \SS_M) \to L^2(M, \SS_M)$ is
 $B$-coercive at infinity. Let $P_B\colon H_{\frac{1}{2}}(\Sigma,\SS_M|_\Sigma)\to B$ be a projection. Moreover, assume that 
$\psi\in L^2(M, \SS_M)$ and $\rho\in H_{\frac{1}{2}}(\Sigma, \SS_M|_\Sigma)$  satisfy
\begin{equation}\label{int_cond} (\psi,\chi)+(\nu\cdot \rho, R\chi)_\Sigma=0\end{equation}
for all $\chi\in \ker\, (D_B)^*$. Then, 
 the boundary value problem
$$\left\{
\begin{array}{rll}
D\phi &=\psi  &\text{on\ } M,\\
(\Id - P_{B})R\phi&= (\Id - P_{B}){\rho}&\text{on\ } \Sigma
\end{array}
\right.
$$
has a solution $\phi\in H_1(M,\SS_M)$ that is unique up to elements of the kernel $\ker\, D_B$.
\end{corollary}

\begin{proof} By Lemma \ref{H1_dommax}, $B\subset H_{\frac{1}{2}}(\Sigma, \SS_M|_\Sigma)$ implies $\dom\, D_B\subset H_1(M,\SS_M)$. We set 
$\tilde{\rho}=\mathcal{E}\rho$. By the Trace Theorem \ref{trace_theorem}, 
$\tilde{\rho}\in H_1(M,\SS_M)$. Moreover, by Lemma \ref{againextequ} the integrability condition \eqref{int_cond} 
implies that $\psi-D\tilde{\rho}\in (\ker\, (D_{B})^*)^\perp$. Hence, together with the Closed Range Theorem there is $\hat{\phi}\in \dom\, D_B\subset H_1(M,\SS_M)$ with $D\hat{\phi}= \psi-D\tilde{\rho}$. Thus, as in the proof of Theorem \ref{intro-bvp} $\phi=\hat{\phi}+\tilde{\rho}$ gives a solution which is now in $H_1(M,\SS_M)$. 
\end{proof}

\begin{remark}In order to give a full generalization of the theory given
in \cite{baer_ballmann_11} it would be interesting to examine the following questions:\\
- Consider general boundary conditions, in particular we would like to identify the image of the extended trace map in Theorem \ref{extended-trace}.\\
- Give a generalization of the definition for elliptic boundary conditions for noncompact boundaries (of bounded geometry) and study them.\\
- Consider, more generally, complete Dirac-type operators as in \cite{baer_ballmann_11}.
\end{remark}

\section{On the boundary condition $B_\pm$}\label{boundcond_B+-}
In this section,  we briefly recall and give some basic facts on $P_\pm$. Some of them can be found in \cite[Section 6]{HMZ02}. Moreover, we prove the claims of Example \ref{ex_bd_cond}.iii. 

\begin{lem}\label{lemmaonP+-} Let $P_\pm\colon L^2(\Sigma, \SS_M|_\Sigma)\to  L^2(\Sigma, \SS_M|_\Sigma)$ be the
 map $\phi\mapsto \frac{1}{2}(\phi \pm \i\nu \cdot \phi)$ and consider $B_\pm:= \{ \phi\in H_{\frac{1}{2}}(\Sigma, \SS_M|_\Sigma ) \ |\ P_\pm\phi=0\}$. Then,  the following hold
\begin{itemize}
 \item[(i)] $P_\pm$ are self-adjoint projections, orthogonal to each other and $\nu P_\pm=P_\pm \nu= \mp \i P_\pm$.
 \item[(ii)] For all $s\in \mathbb{R}$, $P_\pm(\phi)= \frac{1}{2}(\phi \pm i\nu \cdot \phi)$ gives an operator from
 $H_s(\Sigma, \SS_M|_\Sigma)$ to itself such that for all $\phi\in H_s(\Sigma, \SS_M|_\Sigma)$ and 
$\psi\in H_{-s}(\Sigma, \SS_M|_\Sigma)$ we have $(P_+\phi,P_-\psi)_\Sigma=0$ and
 $(P_\pm \phi, \psi)_\Sigma= (\phi, P_\pm \psi)_\Sigma$.
  \item[(iii)] $\widetilde{D}^{\Sigma}P_{\pm}=P_{\mp} \widetilde{D}^{\Sigma}.$
 \item[(iv)] $D_\pm$ (see Example \ref{ex_bd_cond}.iii for the definition) is a closed extension of $D_{cc}$. 
 \item[(v)] $D_\pm= D_{B_\pm}$.
   \item[(vi)] $(D_{B_\pm})^*=D_{B_\mp}$.
 \item[(vii)] Let each connected component of $M$ have a non-empty boundary. Then, $\ker D_{B_\pm} =\{0\}$.
 \end{itemize}

\end{lem}

\begin{proof}
Assertions (i) and (ii) follow directly by simple calculations, and (iii) follows directly from \eqref{d1}.   For (iv) we have by definition of $D_\pm$ (see Example \ref{ex_bd_cond}.iii) that
$D_\pm=D_{\tilde{B}_\pm}$ where $\tilde{B}_\pm=\{\phi\in R(\dom\, D_{\mathrm{max}}) \ |\ P_\pm\phi=0\}$. 

In order to show the closedness of $D_\pm$ we want to apply Lemma \ref{Bmax}. For that, we have to show that $\tilde{B}_\pm$ is closed in $\check{R}$: Let $\phi_i\in \tilde{B}_\pm$ with $\phi_i\to \phi$ in
 $\check{R}$. 
 Then, we get together with Remark \ref{rem_norms}.ii that
\begin{align*}\Vert P_\pm \phi\Vert_{\check{R}}=& \Vert P_\pm (\phi-\phi_i)\Vert_{\check{R}}=\Vert\mathcal{E} P_\pm (\phi-\phi_i)\Vert_{D}\leq \frac{1}{2}\left( \Vert\mathcal{E} (\phi-\phi_i)\Vert_{D} +\Vert\mathcal{E} \nu\cdot (\phi-\phi_i)\Vert_{D}\right)\\
\leq& C\Vert \mathcal{E} (\phi-\phi_i)\Vert_{D} = \Vert \phi-\phi_i\Vert_{\check{R}}\to 0.
\end{align*}

Hence, $P_\pm \phi=0$ and $\phi\in \tilde{B}_\pm.$

For (v), we have clearly that $\dom\, D_{B_\pm}\subset \dom\, D_\pm$. It remains to show that
 any $\phi\in \dom\, D_\pm$ is already in $H_1(M, \SS_M)$. By  Lemma \ref{dense_in_graph_norm}, 
 there is a sequence $\phi_i\in \Gamma_c^\infty(M,\SS_M)$ with $\phi_i\to \phi$ in the graph norm. 
Consider $\mathcal{E}P_\pm R\phi_i$. By the linearity of $\mathcal{E}$, \eqref{ER_ext} and Remark \ref{rem_norms}.ii  we get
\begin{align*} \Vert \mathcal{E}P_\pm R \phi_i\Vert_D=& \Vert \mathcal{E}P_\pm R(\phi_i-\phi)\Vert_D\\
\leq& \frac{1}{2}\left(\Vert \mathcal{E} R(\phi_i-\phi)\Vert_D + \Vert \mathcal{E}(\nu\cdot R(\phi_i-\phi)\Vert_D \right))\leq C \Vert \phi_i-\phi\Vert_D\to 0.
\end{align*}
 Hence, $\psi_i:=\phi_i - \mathcal{E}P_\pm R \phi_i \to \phi$ in the graph norm. Since $\psi_i \in \dom\, D_{B_\pm}$, this implies 
that $\dom\, D_{B_\pm}$ is dense in $\dom\, D_\pm$. Moreover, note that with (iii) and (i) we have 
$$\int_\Sigma \langle R\psi_i, \widetilde{D}^{\Sigma}R\psi_i\rangle ds= \int_\Sigma \langle P_\mp R\psi_i, \widetilde{D}^{\Sigma}P_\mp R\psi_i\rangle
 ds=\int_\Sigma \langle P_\mp R\psi_i, P_\pm \widetilde{D}^{\Sigma} R\psi_i\rangle ds=0.$$

Hence, together with the Lichnerowicz formula in Lemma \ref{extequ}, the bounded geometry, (i) and Lemma \ref{extequ} we get
\begin{align*}
\Vert \psi_i-\psi_j\Vert_{H_1}^2=&\Vert \psi_i-\psi_j\Vert_D^2-\frac{1}{4}\int_M\langle (\s^M+2\i\Omega\cdot)(\psi_i-\psi_j),(\psi_i-\psi_j)\rangle dv\\
&-\frac{n}{2}\int_\Sigma H|R(\psi_i-\psi_j)|^2 ds\\
 \leq&  C\Vert \psi_i-\psi_j\Vert_D^2 \mp \i \frac{n}{2}\int_\Sigma <\nu\cdot R(\psi_i-\psi_j), H  R(\psi_i-\psi_j)>\\
 \leq&  C\Vert \psi_i-\psi_j\Vert_D^2.
\end{align*}
Thus, $\psi_i$ is even a Cauchy sequence in $H_1$ which implies that $\phi$ is already in $H_1(M,\SS_M)$. 
Note that this implies in particular that $B_\pm=\tilde{B}_\pm$. 
For (vi), 
the domain of the adjoint is defined by
\[\dom\, (D_{+})^*=\{\theta\in L^2(M,\SS_M)\ |\ \exists \chi\in L^2(M,\SS_M)\, \forall \psi\in \dom\, D_{+}: 
(\chi, \psi)=(\eta, D\psi) \}.\]
Since, $\Gamma_{cc}^\infty(M,\SS_M) \subset  \dom\, D_{+}$, we get  
$\dom\, (D_{+})^* \subset \dom\, D_{\mathrm{max}}$. Thus,
$$\dom\, (D_{+})^*=\{\theta\in \dom\, D_{\mathrm{max}}\ |\ \forall \psi\in \dom\, D_{+}: 
(D\theta, \psi)=(\theta, D\psi) \}.$$
Due to Lemma \ref{againextequ}, the definition of $\dom\, D_{+}$ and (v), we get
\begin{align*}
 \dom\, (D_{+})^*=&\Big\{\theta\in \dom\, D_{\mathrm{max}}\ |\ \forall \psi\in H_1(M,\SS_M): 
\int_\Sigma \langle \nu\cdot R\theta,  P_-R\psi\rangle ds=0 \Big\}.\end{align*}
By (i) and (ii), we have
\[-\int_\Sigma \langle R\theta, \nu\cdot P_-R\psi\rangle ds=\i\int_\Sigma \langle R\theta, P_-R\psi\rangle ds=\i\int_\Sigma \langle P_-R\theta, R\psi\rangle ds\] and $P_-R\theta\in H_{-\frac12}(\Sigma, \SS_M|_\Sigma)$.
Hence, together with Lemma \ref{H1_dommax} and Lemma \ref{pairing-Sobolev}, 
\begin{align*}
\dom\, (D_{+})^*=&\Big\{\theta\in \dom\, D_{\mathrm{max}}\ |\ \forall \hat{\psi}\in H_{\frac{1}{2}}(\Sigma,\SS_M|_\Sigma): 
\int_\Sigma \langle P_-R\theta, \hat{\psi}\rangle ds=0 \Big\}\\
=&\{\theta\in \dom\, D_{\mathrm{max}}\ |\  P_-R\theta =0 \}=\dom\ D_-.
\end{align*}
The assertion (vii)  is proven as in the closed case \cite[Proof of Corollary 6]{HMZ02}:  Let $\phi\in \ker D_{\pm}$, i.e. $\phi\in\dom\, D_{\mathrm{max}}$, $D\phi=0$ on $M$, and $P_\pm R\phi=0$ on $\Sigma$. Using this, \eqref{L2-structure_mod_boundary}, Lemma \ref{againextequ} 
and (i), we compute
\begin{align*}
 0&=\int_M \langle \phi, \i D\phi\rangle dv - \int_M \langle D\phi, \i\phi\rangle dv  =\int_\Sigma \langle \nu \cdot R\phi, \i R\phi\rangle ds  \\
 &=\int_\Sigma \langle \nu \cdot P_\mp R\phi, \i P_\mp R\phi\rangle ds=\pm \int_\Sigma |R \phi|^2 ds.
\end{align*}
Hence, $R\phi=0$ and $\phi\in \dom\, D_{\mathrm{min}}$, cf. Lemma \ref{norm_equ_0}. But due to the strong unique continuation property of the Dirac operator \cite[Section 1.2]{BBL}, $D_{\mathrm{min}}\phi=0$ implies $\phi=0$.
\end{proof}

\section{Examples and the coercivity condition}\label{section_coer}

In Definition \ref{coer}, we defined when an operator
 $D_B$ is $(\mathrm{dom}\, D_B)$-coercive at infinity. When working with $B$, we will also use the 
short version -- $B$-coercive at infinity. In this passage, we will compare this notion with the one of coercivity at infinity given in \cite[Definition 8.2]{baer_ballmann_11} as cited below and give some examples.

\begin{definition}\label{coerbb}
\cite[Definition 8.2]{baer_ballmann_11}
 $D\colon \dom\, D_{\mathrm{max}} \subset L^2(M,\SS_M)\to L^2(M, \SS_M)$ is coercive at infinity  if there is a compact subset $K\subset M$ and 
 a constant $c>0$ such that $$\Vert D\phi\Vert_{L^2}
 \geq c\Vert \phi\Vert_{L^2},$$ for all $\phi\in \Gamma_c^\infty (M\setminus K, \SS_M)$.
\end{definition}

By \cite[Lemma 8.4]{baer_ballmann_11}, $D$ is coercive at infinity for a closed boundary
 $\Sigma$ if and only if there is a compact subset $K\subset M$ and a constant $c>0$ such that for all
 $\phi\in \Gamma_{cc}^\infty(M\setminus K, \SS_M)$ we have $\Vert D\phi\Vert_{L^2}
 \geq c\Vert \phi\Vert_{L^2}$. For noncompact boundaries, just the 'only if'-direction survives 
since in contrast to closed boundaries there is no compact $K$ such that $\Gamma_c^\infty(M\setminus K,\SS_M)\subset \Gamma_{cc}^\infty(M,\SS_M)$.\\

Before we compare those different coercivity conditions we give some examples:

\begin{example}
 \begin{itemize}
  \item[(i)] By the unique continuation property, the kernel of $D_\mathrm{min}$ is trivial.  Thus, 
together with Lemma \ref{norm_equ_0},  we have that $D$ is $(B=0)$--coercive at infinity if 
and only if there is a constant $c>0$ such that for all $\phi\in \Gamma_{cc}^\infty(M,\SS_M)$
  $$\Vert D\phi\Vert_{L^2} \geq c\Vert \phi\Vert_{L^2}.$$ 
For closed boundaries, this implies coercivity at infinity by  \cite[Lemma 8.4]{baer_ballmann_11} which was cited above. We will see that for closed boundaries also the converse is true, cf. Corollary \ref{cor_coer_closed}.
\item[(ii)] By Lemma \ref{lemmaonP+-}, $\ker D_{B_\pm}=\{0\}$. Thus, $D$ is $B_\pm$-coercive
 at infinity if and only if there is a constant $c>0$ such that $$\Vert D\psi\Vert_{L^2}\geq c\Vert \psi\Vert_{L^2}$$
 for all $\psi\in H_1(M,\SS_M)$ with $P_\pm R\psi =0$. In particular, this implies $(B=0)$-coercivity at infinity. More generally, if $B_1\subset B_2$ and $\ker D_{B_1}=\ker D_{B_2}$, then $B_2$-coercivity at infinity implies $B_1$-coercivity at infinity.
 \end{itemize}
\end{example}

\begin{lemma}\label{equiv_coer_easydir}
 Let $D$ be coercive at infinity, and let $B$ be a boundary condition. Assume that 
$\dom\, D_B\cap (\ker D_B)^\perp\subset H_1(M,\SS_M)$ and that  the $H_1$-norm and the graph norm are
 equivalent on $\dom\, D_B\cap (\ker D_B)^\perp$. Then, $D$ is $B$-coercive at infinity.
\end{lemma}

\begin{proof} Since $D$ is coercive at infinity, there is a compact subset $K\subset M$ and a constant $c>0$ such 
that $\Vert D\phi\Vert_{L^2}\geq c\Vert \phi\Vert_{L^2}$ for all $\phi\in \Gamma_c^\infty(M\setminus K, \SS_M)$. 
Assume that $D$ is not $B$-coercive at infinity. Then, there is a sequence
 $\phi_i\in \dom\, D_B\cap (\ker D_B)^\perp$ with $\Vert \phi_i\Vert_{L^2}=1$ and $\Vert D\phi_i\Vert_{L^2}\to 0$. By
 equivalence of the norms, $\phi_i$ is also bounded in $H_1$. This implies $\phi_i \to \phi$ weakly in $H_1$ and, thus, locally strongly in $L^2$. Moreover, $D\phi=0$. 
 Together with $\phi_i\perp \ker D_B$, this implies $\phi=0$. Thus, for 
each compact subset $K'\subset M$ we have $\int_{K'} |\phi_i|^2 dv\to 0$ as $i\to \infty$. Let 
$\eta\colon M \to [0,1]$ be a cut-off function and $K'$ be a compact subset such that 
$K\subset K' \subset M$ and $\eta=0$ on $K$, $\eta=1$  on $M\setminus K'$ and $|d \eta|\leq a$ for a constant $a>0$ big enough.  
Then, $\supp\, (\eta\phi_i)\subset M\setminus K$, $\Vert D(\eta \phi_i)\Vert_{L^2}\leq a\Vert \phi_i\Vert_{L^2(K')}+\Vert D\phi_i\Vert_{L^2}\to 0$ 
and
 $$1\geq \Vert\eta \phi_i\Vert_{L^2}\geq \Vert \phi_i \Vert_{L^2}- \Vert (1-\eta)\phi_i
\Vert_{L^2}\geq 1-\Vert \phi_i\Vert_{L^2(K')}\to 1.$$ By Lemma \ref{dense_in_graph_norm}, we can choose a sequence $(\phi_{ij})_j\subset \Gamma_c^\infty(M,\SS_M)$ with $\phi_{ij}\to \phi_i$ in the graph norm as $j\to \infty$. Then, $\eta\phi_{ij}\to \eta \phi_i$ in the graph norm and $\supp\, (\eta\phi_{ij})\in M\setminus K$. Thus, we can find $j=j(i)$ such that $\Vert D(\eta \phi_{ij(i)})\Vert_{L^2}\to 0$ and $\Vert \eta \phi_{ij(i)}\Vert_{L^2}\to 1$ as $i\to \infty$. But this contradicts  the assumption that $D$ is coercive at infinity.
\end{proof}

From the last Lemma and Lemma \ref{equiv_H1_D} we obtain immediately

\begin{corollary}\label{easy_dir_cor} If $D$ is coercive at infinity and $B\subset H_{\frac{1}{2}}(\Sigma, \SS_M|_\Sigma)$, then $D$ is $B$-coercive at infinity.
\end{corollary}

Next we give some (very restrictive) conditions that are sufficient to prove that $B$-coercivity at infinity implies coercivity at infinity. Those additional assumptions are needed to make sure that the $\phi_i$ appearing in Definition \ref{coerbb} are in $\dom\, D_B$.

\begin{lemma}\label{equiv_coer}Let $B$ be a boundary condition with $B\subset H_\frac{1}{2}(\Sigma, \SS_M|_\Sigma)$. Assume that there exists a compact subset $K'\subset M$ with $\Gamma_c^\infty(M\setminus K', \SS_M)\subset \dom\, D_B$.  
If $D\colon \mathrm{dom}\, D_B\subset L^2(\Sigma, \SS_M|_\Sigma) \to L^2(\Sigma, \SS_M|_\Sigma)$ has a finite
dimensional kernel  and $D$ is $B$-coercive at infinity, then $D$ is coercive at infinity.
\end{lemma}

\begin{proof}Assume that $D$ is not coercive at infinity. Then, for all compact subsets $K\subset M$ there exists a sequence $\phi_i\in \Gamma_c^\infty(M\setminus K, \SS)$ with $\Vert \phi_i\Vert_{L^2}=1$ and $\Vert D\phi_i\Vert_{L^2}\to 0$.  We choose $K$ such that $K'\subset K$. Then, all those $\phi_i\in \dom\, D_B$. Thus, $\phi_i\to \phi\in \dom\, D_B$ weakly in the graph norm of $D$, $\phi\in \ker\, D_B$ and $\phi=0$ on $K$. We decompose $\phi_i=\phi_i^k +\phi_i^\perp$ where $\phi_i^k \in \ker\, D_B$ and  $\phi_i^\perp \in \left( \ker D_B\right)^\perp$. Then $\Vert D\phi_i^\perp\Vert_{L^2}\to 0$. Moreover, we assume that the kernel is finite dimensional, i.e. $\phi_i^k=\sum_{j=1}^l a_{ij}\psi_j$ where the $\psi_j$'s form an orthonormal basis of $\mathrm{ker}\, D_B$. Thus, $\Vert
\phi_i^k\Vert_{L^2}^2=\sum_{j=1}^l|a_{ij}|^2$. Assume now that $\Vert \phi_i^\perp\Vert_{L^ 2}\to 0$. Then $\phi_i^\perp\to 0$ in the graph norm. But $\Vert \phi_i\Vert_{L^2}=1$. This implies that there is at least one $j\in \{ 1, \ldots, l\}$ with $|a_{ij}|$ is bounded away from zero for almost all $i$, i.e. $\phi$ cannot be zero everywhere. Since $\phi$ is zero on $K$, this is a contradiction to the unique continuation principle. Thus, the assumption was wrong and there exists $c>0$ with $\Vert \phi_i^\perp\Vert_{L^2}>c$ and $D$ is not $B$-coercive at infinity.
\end{proof}

Note that the assumption on the existence of $K'$ is very restrictive. If the boundary is closed, it is automatically satisfied and we get the corollary below. If the boundary is noncompact, for a general $\dom\, D$ e.g. for the minimal domain of $D$, it is not true. 
But there are also examples for manifolds with noncompact boundary and closed extension of $D_{cc}$ where the  assumptions of the last Lemma are satisfied:

\begin{example} Let $(\Sigma, h)$ be a complete Riemannian $\spin$ manifold.
  Let $M_\infty=\Sigma\times \R$ and $M=\Sigma\times [0,\infty)$ be equipped with product metric $h+dt^2$. Both manifolds are of bounded geometry. Since $M_\infty$ is complete with no boundary, the Dirac operator on $M_\infty$ is essentially self-adjoint. Assume that the Dirac operator on $M_\infty$ is invertible.

  Let $K'\subset M_\infty$ be a compact subset that intersects $\Sigma\times \{0\}$ in a subset of non-zero measure. Define $\mathcal{L}$ to be the linear span of $\Gamma_c^\infty(M\setminus K', \SS_M)\cup \Gamma_{cc}^\infty(M,\SS_M)$ and $\dom\, D_B\colon= \overline{\mathcal{L}}^{\Vert. \Vert_D}$. Then, $B=\overline{\Gamma_c^\infty(\Sigma\setminus K', \SS_M|_\Sigma)}^{\Vert. \Vert_{\check{R}}}$. Note that by construction $\dom\, D_B$ is the domain of a closed extension of $D_{cc}$. But it is honestly smaller than $\dom\, D_{\rm max}$ since all $\phi\in B$ have to vanish on  $\Sigma\cap K'$. In particular, by the strong unique continuation property of $D$ \cite[Section 1.2]{BBL}  $D_B\colon \dom\, D_B \to L^2(M,\SS_M)$ has trivial kernel.
  
 It remains to show that $D_B$ is $B$-coercive at infinity, i.e. there is $c>0$ such that for all $\phi\in \mathcal{L}$ we have $\Vert D\phi\Vert_{L^2}\geq c\Vert \phi\Vert_{L^2}$. We will show this by contradiction, that is, we assume that there is a sequence $\phi_i\in \mathcal{L}$ with $\Vert \phi_i\Vert_{L^2}=1$ and $\Vert D\phi_i\Vert_{L^2}\to 0$. We will construct a sequence of spinors on $M_\infty$. Let $\tilde{\phi}_i$ be obtained from $\phi_i$ by reflection along $\Sigma$. Clearly, $\tilde{\phi}_i\in L^2(M_\infty, \SS_{M_\infty})$. Moreover, note that $\tilde{\phi}_i$ is everywhere continuous.  Let $\nu$ be the inward  normal vector field of $M$.
 For $\psi\in \Gamma_c^\infty(M_\infty, \SS_{M_\infty})$ we can estimate using \eqref{L2-structure_mod_boundary}
 \begin{align*}
  |(\tilde{\phi}_i,& D\psi)_{L^2(M_\infty)}|= \left|\int_{\Sigma\times (0,\infty)} \< \tilde{\phi}_i, D\psi\> + \int_{\Sigma\times (-\infty,0)} \< \tilde{\phi}_i, D\psi\>\right|\\
  = &\left|\int_{\Sigma\times (0,\infty)} \< D\tilde{\phi}_i, \psi\>  + \int_{\Sigma} \< \nu\cdot \tilde{\phi}_i|_\Sigma, \psi|_\Sigma\> +  \int_{\Sigma\times (-\infty,0)} \< D\tilde{\phi}_i, \psi\> + \int_{\Sigma} \< -\nu\cdot \tilde{\phi}_i|_\Sigma, \psi|_\Sigma\> \right|\\
  \leq& 2\Vert D\phi_i\Vert_{L^2(M)} \Vert \psi\Vert_{L^2(M_\infty)}\to 0.
 \end{align*}

In particular this means that $\tilde{\phi}_i\in H_1(M_\infty, \SS_{M_\infty})$ and that $\Vert D\tilde{\phi}_i\Vert_{L^2(M_\infty)}\to 0$ while $\Vert \tilde{\phi}_i\Vert_{L^2(M_\infty)}=2$. This gives a contradiction to the invertibility of the Dirac operator on $M_\infty$. 
\end{example}

\begin{corollary}\label{cor_coer_closed} Let the boundary $\Sigma$ be closed.
 If $B$ is an elliptic boundary condition as defined in \cite[Definition 7.5]{baer_ballmann_11}, 
$B$-coercivity at infinity implies coercivity at infinity. In particular, $D$ is $(B=0)$-coercive at infinity if and 
only if it is coercive at infinity.
\end{corollary}
\begin{proof}
If the boundary is closed and $B$ is elliptic, $D_B$ has a finite kernel \cite[Theorem 8.5]{baer_ballmann_11}.
The rest of the assumption in Lemma \ref{equiv_coer} is trivially fulfilled which gives the first claim. The rest follows with Corollary \ref{easy_dir_cor}.
\end{proof}

For closed boundaries and spin manifolds, assuming uniformly positive scalar curvature at infinity is a sufficient 
condition to have that $D$ is coercive at infinity, see \cite[Example 8.3]{baer_ballmann_11}. For noncompact boundaries,  we obtain the following 
\begin{lem}
\begin{itemize}
\item[(i)] If $\frac 12 \s^M +\i\Omega\cdot$ is a positive operator, the Dirac operator $D$ is $(B=0)$-coercive at infinity.
\item[(ii)] If $\frac 12 \s^M +\i\Omega\cdot$ is a positive operator and $H\geq 0$, the Dirac operator $D$ is $B_\pm$-coercive at infinity.
\end{itemize}
 \end{lem}
 
 \begin{proof} Let $c>0$ such that $\frac 12 \s^M +\i\Omega\cdot\geq 2c$. The Lichnerowicz formula 
\eqref{Lich} and Lemma \ref{extequ} give
\begin{align*}
\Vert D\phi\Vert_{L^2}^2&=\Vert \nabla \phi\Vert_{L^2}^2+\int_{M}
\frac{\s^M}{4}|\phi|^2 dv  + \int_{M} \frac{\i}{2}<\Omega\cdot\phi, \phi> dv -
\int_{\Sigma} \langle R\phi,\widetilde{D}^{\Sigma}(R\phi)\rangle ds\\
& +\frac{n}{2}\int_{\Sigma} H|R\phi|^2ds\geq c\Vert \phi\Vert_{L^2}^2 -\int_{\Sigma} \langle R\phi,
\widetilde{D}^{\Sigma}(R\phi)\rangle ds +\frac{n}{2}\int_{\Sigma} H|R\phi|^2ds,
\end{align*}
for all $\phi\in H_1(M, \SS_M)$. Then (i) follows directly with Lemma \ref{norm_equ_0}. For (ii), let now $H\geq 0$ and $R\phi\in B_\pm$.
Then, together with Lemma \ref{lemmaonP+-}, it implies
\begin{align*}
\Vert D\phi\Vert_{L^2}^2&\geq c\Vert \phi\Vert_{L^2}^2 -\int_{\Sigma} \langle R\phi,
\widetilde{D}^{\Sigma}(R\phi)\rangle ds=c\Vert \phi\Vert_{L^2}^2 -\int_{\Sigma} \langle P_\mp R\phi,
 \widetilde{D}^{\Sigma}(P_\mp R\phi)\rangle\\
 &= c\Vert \phi\Vert_{L^2}^2 -\int_{\Sigma} \langle P_\mp R\phi,
 P_\pm\widetilde{D}^{\Sigma}(R\phi)\rangle= c\Vert \phi\Vert_{L^2}^2.
\end{align*}
\end{proof}

\section{ $\Spinc$ Reilly inequality on possibly open boundary domains}\label{Reilly-sec}

In this section, we shortly review the spinorial Reilly inequality. This inequality together with those boundary value problems
 discussed in Section \ref{boundary_values} will be the main ingredient in the proof of Theorem \ref{main}.

\begin{thm} {\bf $\Spinc$ Reilly inequality.}
\label{Reilly} For all $\psi\in H_1(M,\SS_M)$, we have
\begin{eqnarray}\label{inequalityReilly}
\int_{\Sigma}\Big(\langle \widetilde{D}^{\Sigma}\psi,\psi\rangle-\frac{n}{2}
H|\psi|^2\Big)ds\geq\int_{M}\Big(\frac{1}{4} \s^M
 \vert\psi\vert^2 +\frac 12\langle \i\Omega\cdot\psi, \psi\rangle -\frac{n}{n+1}|D\psi|^2\Big)dv,
\end{eqnarray}
where $dv$ (resp. $ds$) is the Riemannian volume form of $M$ (resp. $\Sigma$).
Moreover,
equality occurs if and only if the spinor field $\psi$ is a twistor-spinor,
i.e. if and only if $P\psi=0$, 
where $P$ is the twistor operator
acting on $\SS_M$ and is locally given by $P_X\psi=\nabla_X\psi+\frac{1}{n+1}X\cdot D\psi$ for all $X \in\Gamma(TM)$. 
\end{thm}
\begin{proof} The inequality is proved for $\psi\in \Gamma_c^\infty(M, \SS_M)$ analogously as in the compact $\Spin$
 case \cite[(17)]{HMZ1}. For the convenience of the reader, we will shortly recall it here. 
Then for all $\psi\in H_1(M,\SS_M)$ the claim follows using the Trace Theorem \ref{trace_theorem} in the same way as in Lemma \ref{extequ}: We define $1$-forms $\alpha$ and $\beta$ on $M$ by 
$\alpha(X) = \langle X\cdot D\psi, \psi\rangle$ and $\beta(X) =
\langle\nabla_X\psi, \psi\rangle$ for all $X
\in \Gamma^\infty(TM)$. Then $\alpha$ and $\beta$ 
satisfy
$$\delta\alpha = \langle D^2\psi, \psi\rangle - \vert D \psi\vert^2, \quad
\delta\beta = 
- \langle\nabla^*\nabla\psi, \psi\rangle + \vert\nabla\psi\vert^2.$$
Applying the divergence theorem with \eqref{sl} and \eqref{diracgauss}, we get
\begin{eqnarray}\label{div}
\int_{\Sigma}\Big(\langle\widetilde{D}^{\Sigma} \psi, \psi\rangle -\frac n2
H\vert\psi\vert^2 \Big) ds = \int_{M} \Big(\vert\nabla
\psi\vert^2 - \vert D\psi\vert^2 +\frac 14 \s^M\  \vert\psi\vert^2 + \frac{\i}{2} \langle\Omega\cdot\psi, \psi\rangle\Big) dv.
\end{eqnarray}
On the other hand, for any spinor field $\psi$ we have 
\begin{eqnarray}
 \label{twistor}
\vert\nabla\psi\vert^2 =  \vert P\psi\vert^2 + \frac{1}{n+1} \vert D\psi\vert^2.
\end{eqnarray}
Combining the identities \eqref{twistor}, and \eqref{div} and $\vert
P\psi\vert^2 \geq 0$, the result follows. Equality holds if and only if 
$\vert P\psi\vert^2 = 0$, i.e. the spinor $\psi$ is a twistor spinor.
\end{proof}

\section{A lower bound for the first nonnegative eigenvalue of the Dirac operator
on the boundary}\label{proofmain}


In this section, we prove Theorem \ref{main}. For that we won't follow the original proof given in \cite{HMZ1} due to 
our problems concerning the $\mathrm{APS}$-boundary conditions as remarked at the end of Example \ref{ex_bd_cond}.iv. But we will use $B_\pm$ as given in Example \ref{ex_bd_cond}.iii. 

\begin{proof}[Proof of Theorem \ref{main}] 
Since $\Sigma$ is of bounded geometry, $\widetilde{D}^{\Sigma}: H_1(\Sigma, \SS_M|_\Sigma)\to L^2(\Sigma, \SS_M|_\Sigma)$ is self-adjoint and, hence,  
$\lambda_1$ is an eigenvalue or in the essential spectrum of  $\widetilde{D}^{\Sigma}$. In both cases, there is
 a sequence $\phi_i\in H_1(\Sigma, \SS_M|_\Sigma)$ with $\Vert \phi_i\Vert_{L^2(\Sigma)}=1$ and
 $\Vert(\widetilde{D}^{\Sigma} -\lambda_1)\phi_i\Vert_{L^2(\Sigma)}\to 0$. Then, $\phi_i\to \phi$ weakly in $L^2(\Sigma, \SS_M|_\Sigma)$. (In case that $\phi\neq 0$, then $\phi$ is an eigenspinor of $\widetilde{D}^{\Sigma}$ to the eigenvalue $\lambda_1$ otherwise $\lambda_1$ is in the essential spectrum of $\widetilde{D}^{\Sigma}$). We assumed that $D$ is $B_-$-coercive at infinity 
(everything which follows is also true when assuming 
$B_+$-coercivity at infinity when switching the signs). Then  by  Lemma \ref{coercive_closed}, 
 the range of
 $D_{B_-}$ is closed. Moreover, from Lemma \ref{lemmaonP+-} we have $\ker\, (D_{B_-})^*= \ker\, D_{B_+}=\{0\}$. Thus, 
due to Corollary \ref{bvp-H1} for each $i$
 there exists a unique $\Psi_i\in H_1(M,\SS_M)$ with $D\Psi_i=0$ and $P_+R\Psi_i=P_+\phi_i.$ Using 
Theorem \ref{Reilly} and $\s^M+2\i\Omega\cdot \geq 0$, we obtain
\begin{align*}
0\leq \int_{\Sigma} \left( \langle \widetilde{D}^{\Sigma}R\Psi_i,R\Psi_i\rangle
-\frac{n}{2}H|R\Psi_i|^2\right) ds.
\end{align*}
Moreover,
\begin{align*}(\widetilde{D}^{\Sigma} (P_+R\Psi_i +P_-R\Psi_i), P_+R\Psi_i +P_-R\Psi_i)_\Sigma&=
 ( \widetilde{D}^{\Sigma} P_+ R\Psi_i, P_-R\Psi_i)_\Sigma + ( \widetilde{D}^{\Sigma} P_- R\Psi_i, P_+R\Psi_i)_\Sigma\\&=
 ( \widetilde{D}^{\Sigma} P_+ R\Psi_i, P_-R\Psi_i)_\Sigma + ( P_- R\Psi_i,  \widetilde{D}^{\Sigma} RP_+\Psi_i)_\Sigma,\end{align*}
where we used Lemma \ref{lemmaonP+-} and that $\widetilde{D}^{\Sigma}$ is self-adjoint on $H_1(\Sigma, \SS_M|_\Sigma)$. Hence, summarizing we get that

\begin{align*}
\frac{n}{2}\int_{\Sigma} H|R\Psi_i|^2ds &\leq 2\Re \int_{\Sigma}  \langle \widetilde{D}^{\Sigma} P_+ R\Psi_i,P_-R\Psi_i\rangle
 ds=2\Re \int_{\Sigma}  \langle P_-\widetilde  D \phi_i,P_-R\Psi_i\rangle
 ds\\
&\leq 2\Re \int_{\Sigma}  \langle P_-(\widetilde{D}^{\Sigma}-\lambda_1)  \phi_i,P_-R\Psi_i\rangle
 ds+ 2\lambda_1 \Re \int_{\Sigma}  \langle P_- \phi_i,P_-R\Psi_i\rangle
 ds.
\end{align*}

Using  $2\Re \int_{\Sigma}  \langle P_- \phi_i,P_-R\Psi_i\rangle
 ds\leq \Vert P_- \phi_i \Vert_{L^2(\Sigma)}^2 +\Vert P_-R\Psi_i\Vert_{L^2(\Sigma)}^2 $ and  $\lambda_1\geq 0$, 
we obtain
\begin{align*}
\frac{n}{2}\inf_\Sigma H \Vert R\Psi_i\Vert_{L^2(\Sigma)}^2&\leq 2\Vert (\widetilde{D}^{\Sigma}-\lambda_1)  \phi_i\Vert_{L^2} \Vert R\Psi_i\Vert_{L^2}+ \lambda_1 (\Vert P_- \phi_i \Vert_{L^2(\Sigma)}^2 +\Vert P_-R\Psi_i\Vert_{L^2(\Sigma)}^2).
\end{align*}
Moreover, $(\widetilde{D}^{\Sigma} P_\pm \phi_i, P_\mp \phi_i)=(P_\mp (\widetilde{D}^{\Sigma} -\lambda_1) \phi_i, P_\mp \phi_i) + 
\lambda_1 \Vert P_\mp \phi_i\Vert_{L^2}^2$. Since $\widetilde{D}^{\Sigma}$ is self-adjoint, 
$\Re (\widetilde{D}^{\Sigma} P_+\phi_i, P_-\phi_i)=  \Re (\widetilde{D}^{\Sigma} P_-\phi_i, P_+\phi_i)$. Thus, together with
 \[|(P_\mp (\widetilde{D}^{\Sigma} -\lambda_1) \phi_i, P_\mp \phi_i)|\leq \Vert  (\widetilde{D}^{\Sigma} -\lambda_1) \phi_i\Vert_{L^2}\Vert\phi_i\Vert_{L^2}\to 0\] as $i\to \infty$, 
 this implies that  $\lim_{i\to \infty} \Vert P_-\phi_i\Vert_ {L^2}= \lim_{i\to\infty} \Vert P_+\phi_i\Vert_ {L^2} =\frac{1}{2}$ for $\lambda_1\neq 0$. 
Hence, for certain $\epsilon_i$ with $\epsilon_i\to 0$ as $i\to\infty$
\begin{align*}
\frac{n}{2}\inf_\Sigma H \Vert R\Psi_i\Vert_{L^2(\Sigma)}^2&\leq 2\Vert (\widetilde{D}^{\Sigma}-\lambda_1)  \phi_i\Vert_{L^2} \Vert R\Psi_i\Vert_{L^2}+ \lambda_1 (\Vert P_+ \phi_i \Vert_{L^2(\Sigma)}^2 +\epsilon_i +\Vert P_-R\Psi_i\Vert_{L^2(\Sigma)}^2)\\
&\leq 2\Vert (\widetilde{D}^{\Sigma}-\lambda_1)  \phi_i\Vert_{L^2} \Vert R\Psi_i\Vert_{L^2}+ \lambda_1 (\Vert P_+ R\Psi_i \Vert_{L^2(\Sigma)}^2 +\epsilon_i +\Vert P_-R\Psi_i\Vert_{L^2(\Sigma)}^2)\\
&\leq 2\Vert (\widetilde{D}^{\Sigma}-\lambda_1)  \phi_i\Vert_{L^2} \Vert R\Psi_i\Vert_{L^2}+ \lambda_1 (\Vert R \Psi_i \Vert_{L^2(\Sigma)}^2 +
\epsilon_i).
\end{align*}
Hence, $$\frac{n}{2}\inf_\Sigma H \leq 2\Vert (\widetilde{D}^{\Sigma}-\lambda_1)  \phi_i\Vert_{L^2} \Vert 
R\Psi_i\Vert_{L^2}^{-1}+ \lambda_1 (1+\epsilon_i \Vert R\Psi_i\Vert_{L^2}^{-2}).$$
With $ \Vert R\Psi_i\Vert_{L^2}\geq  \Vert P_+ R\Psi_i\Vert_{L^2}=\Vert P_+ \phi_i\Vert_{L^2} \to \frac{1}{2}$, we finally get for $i\to\infty$ 
\[ \frac{n}{2}\inf_\Sigma H \leq \lambda_1.\]

Next we collect all conditions that have to be fulfilled to obtain the equality $ \frac{n}{2}\inf_\Sigma H = \lambda_1$:
\begin{itemize}
 \item[(1)] From the spinorial Reilly Inequality \eqref{inequalityReilly}, $\int_M |P\Psi_i|^2dv\to  0$ which implies together with $D\Psi_i=0$ that $\int_M |\nabla \Psi_i|^2dv\to  0$.
\item[(2)] $\int_M \s^M|\Psi_i|^2+2i\langle \Omega\cdot \Psi_i, \Psi_i\rangle dv\to 0$ 
\item[(3)] $\Vert \phi_i - R\Psi_i\Vert_{L^2(\Sigma)}\to 0$
\item[(4)] $\int_{\Sigma} (H-\inf_\Sigma H) |R\Psi_i|^2 ds \to 0$.
\end{itemize}
In case that $\lambda_1$ is an eigenvalue of $\widetilde{D}^{\Sigma}$ with eigenspinor $\phi$, one can
 choose $\phi_i=\phi$ for all $i$. Then $\Psi_i=:\Psi$ for all $i$ and those equality
 conditions reduce to $\phi=R\Psi$, $\Psi$ is a parallel spinor on $M$, $H$ is constant 
and $\int_M \s^M|\Psi|^2+2i\langle \Omega\cdot \Psi, \Psi\rangle dv= 0$.
\end{proof}

\end{document}